\documentclass{siamart171218}
\usepackage{amsmath}
\usepackage{amsfonts}
\usepackage{amssymb}
\usepackage{graphicx}
\usepackage{subcaption}
\usepackage{color}
\usepackage{bm}

\allowdisplaybreaks

\newtheorem{remark}[theorem]{Remark}
\newtheorem{assumption}[theorem]{Assumption}

\newcommand{\be}{\begin{equation}}
\newcommand{\ee}{\end{equation}}
\newcommand{\bea}{\begin{eqnarray}}
\newcommand{\eea}{\end{eqnarray}}
\newcommand{\beas}{\begin{eqnarray*}}
	\newcommand{\eeas}{\end{eqnarray*}}

\newcommand{\Grad}{\ensuremath{\nabla}}

\newcommand{\norm}[2][]{||#2||_{#1}}
\newcommand{\Div}[1]{\nabla \cdot #1}
\newcommand{\dt}{\Delta{t}}
\newcommand{\bstar}[3]{b^*\left(#1,#2,#3\right)}

\input{tcilatex}

\usepackage{url}

\begin{document}
	
	\title{An Artificial Compression Reduced Order Model \thanks{Submitted to the editors 2/23/2019
	\funding{The research of first, third, and fourth authors was partially supported by 
		NSF grants DMS1522267, 1817542 and CBET 1609120. The research of the second author was partially supported by NSF DMS-1821145.}}}

	\author{Victor DeCaria\thanks{Department of Mathematics,  
			University of Pittsburgh,
			Pittsburgh, PA 15206
			({\tt vpd7@pitt.edu}, {\tt mem226@pitt.edu}, {\tt wjl@pitt.edu}).} \and Traian Iliescu\thanks{Department of Mathematics, Virginia Tech, Blacksburg, VA 24061-0123 ({\tt iliescu@vt.edu}).} \and William Layton\footnotemark[2] \and Michael McLaughlin\footnotemark[2] \and Michael Schneier\thanks{Corresponding author. Department of Mathematics,  
			University of Pittsburgh,
			Pittsburgh, PA 15206 ({\tt mhs64@pitt.edu}).}}
	
	\headers{Artificial Compression Reduced Order Model}{V. DeCaria, T. Iliescu, W. Layton, M. McLaughlin, and M.Schneier}
	\maketitle
	
	\begin{abstract}
		We propose a novel artificial compression, reduced order model (AC-ROM) for the numerical simulation of viscous incompressible fluid flows.
		The new AC-ROM provides approximations not only for velocity, but also for pressure, which is needed to calculate forces on bodies in the flow and to connect the simulation parameters with pressure data. 
		The new AC-ROM does not require that the velocity-pressure ROM spaces satisfy the inf-sup (Ladyzhenskaya-Babuska-Brezzi) condition and
		its basis functions are constructed from data that are not required to be weakly-divergence free.
		We prove error estimates for the reduced basis discretization of the AC-ROM.
		We also investigate numerically the new AC-ROM in the simulation of a two-dimensional flow between offset cylinders.
	\end{abstract}

	
	
	\section{Introduction}
	
	We consider the {\it Navier-Stokes equations (NSE)} with no-slip boundary conditions:
	\begin{equation}\label{eqn:nse-1}
	\begin{aligned}
	&u_t + u\cdot\nabla u + \nabla p - \nu\Delta u =  f,\ \text{and } \nabla \cdot u =  0,\ \text{in} \ \Omega \times (0,T]   \\
	&u   =  0, \ \text{on} \ \partial\Omega \times (0,T], \ \text{and } u(x,0)  =  u_0(x), \ \text{in} \ \Omega.  \\
	\end{aligned}
	\end{equation}
	Here $u$ is the velocity, $f$ is the known body force, $p$ is the pressure, and $\nu$ is the kinematic viscosity.
	For the past three decades, {\it reduced order models (ROMs)} have been successfully used in the numerical simulation of fluid flows modeled by the NSE~\eqref{eqn:nse-1} ~\cite{2017arXiv171109780A,2017arXiv171003569F,haasdonk2008reduced, hesthaven2015certified,hinze2005proper,HLB96,noack2011reduced,quarteroni2015reduced,RAMBR17,veroy2005certified}.
	The ROM construction is similar to  the full finite element approximation except we seek a solution in a low dimensional ROM space $X_{R}$ using the basis $\{ \varphi_{i}\}_{i=1}^{R}$.
	These basis functions are often assumed to be weakly divergence-free. 
	This assumption holds true, for example, if the ROM basis functions are constructed from data from a NSE  discretization with finite element velocity-pressure pairs that satisfy the inf-sup (Ladyzhenskaya-Babuska-Brezzi (LBB)) condition.
	In this case, the pressure drops out from the ROM, which yields approximations only for the velocity field: such as for the backward Euler method
\begin{equation}
\Big(\frac{u^{n+1}_R - u^{n}_R}{\Delta t}, \varphi \Big) + b^{\ast}(u_R^{n} , u^{n+1}_R ,\varphi) +\nu (\nabla u^{n+1}_R, \nabla \varphi)  
=( f^{n+1}, \varphi), \quad \forall \varphi \in X_R ,
\label{eqn:g-rom}
\end{equation}
	where
	$
	b^{\ast}(w,u,v):=\frac{1}{2}(w\cdot\nabla u,v)-\frac{1}{2}(w\cdot\nabla v,u) ,
	\ \ \forall u,v,w\in [H^1(\Omega)]^d \, ,
	$
	and the superscript denotes the timestep number. 
	
	We emphasize, however, that even when the pressure is not required in the ROM, one may still need a ROM pressure approximation.
	This happens, for example, in fluid-structure interaction problems, if drag and lift coefficients need to be computed, or if the residual has to be calculated~\cite{caiazzo2014numerical}.
	Another practical issue with velocity only ROMs is that internal (industrial) flows will often have reliable pressure data, but little to no velocity data. 
	A velocity only ROM will be unable to incorporate pressure data to improve accuracy, calibrate the model, or check if a control loop is functioning properly.
	
	When a ROM pressure approximation is required, there are two main approaches that are currently used: \\[-0.3cm]
	
	{\it (I) Inf-Sup/LBB Condition}: \ 
	In the first approach, the velocity and pressure ROM approximations satisfy the inf-sup/LBB condition:
	\begin{eqnarray}
	\inf_{q_{M} \in Q_{M}} \sup_{v_{R} \in X_{R}}
	\frac{(\nabla \cdot v_{R} , q_{M})}{\| \nabla v_{R} \| \, \| q_{M} \|}
	\geq \beta_{is} 
	> 0 \, .
	\label{eqn:inf-sup}
	\end{eqnarray}
	This approach has been extensively developed in the {\it reduced basis method (RBM)} community over the past decade~\cite{hesthaven2015certified,quarteroni2015reduced}.
	This approach yields accurate ROM approximations for both velocity and pressure and eliminates the spurious numerical instabilities in the pressure approximation that are often generated by ROMs that do not satisfy the inf-sup condition.
	Furthermore, rigorous error estimates are proven for the LBB conforming ROM approximations.
	The RBM has been successfully used in numerous scientific and engineering applications~\cite{hesthaven2015certified,quarteroni2015reduced}.
	\textit{However, enforcing the inf-sup condition~\eqref{eqn:inf-sup} is significantly more challenging for ROMs than for finite elements}.
	Indeed, in the finite element context, the approximation spaces (e.g., piecewise quadratic for the velocity and piecewise linear for the pressure, i.e., the Taylor-Hood element) are specified beforehand and the corresponding discrete inf-sup condition can be investigated {\it a priori}.
	In the ROM context, on the other hand, the approximation spaces are problem-dependent -- they are  known only after the underlying finite element simulations (or the actual physical experiments) have been carried out.
	Thus, in a ROM context, the inf-sup condition needs to be enforced for each problem separately.
	In the RBM context, this is generally achieved by enriching the ROM basis with supremizers, which need to be computed in the offline stage.
	Thus, in realistic fluid flow applications (e.g., the NSE at high Reynolds numbers), enforcing the inf-sup condition can be prohibitively expensive (see, e.g., Sections 4.2.2 and 4.2.3 in~\cite{ballarin2015supremizer}).
	\\[-0.3cm]
	
	{\it (II) Pressure Poisson Equation}: \ 
	In the second approach to generate ROM approximation for the pressure, the available ROM velocity approximation is used to solve a pressure Poisson equation for the ROM pressure approximation 
	\begin{equation}
	\Delta p_{M} 
	=  - \nabla\cdot((u_{R} \cdot \nabla) u_{R}) 
	\quad \mbox{in } \Omega\,,
	\label{eqn:pressure-poisson}
	\end{equation}
	which is obtained by taking the divergence of the NSE~\eqref{eqn:nse-1}. 
	This approach has been used in, e.g.,~\cite{akhtar2009stability,caiazzo2014numerical,noack2005need}.
	We note that this approach faces several significant challenges:
	We emphasize that  the Poisson equation~\eqref{eqn:pressure-poisson} is {\it not valid anymore if the ROM basis functions are not weakly divergence-free}.  
	This is the case, for example, if the ROM basis functions are built from data from NSE discretizations with finite element velocity-pressure pairs that do not satisfy the inf-sup/LBB condition, e.g., when the artificial compression, penalty, or projection methods are used~\cite{guermond2006overview}. 
	Furthermore, the boundary conditions for~\eqref{eqn:pressure-poisson} are not clear.
	Finally, the numerical investigation in~\cite{caiazzo2014numerical} showed that even when weakly divergence-free snapshots were used, the ROMs that solve the pressure Poisson equation~\eqref{eqn:pressure-poisson} were less competitive in terms of numerical accuracy and computational efficiency.  
	\\[-0.3cm]
	
	{\it (III) Novel AC-ROM}: \ 
	The {\it artificial compression (AC)} method and related approaches (e.g., the penalty and projection methods) have found significant success in the CFD community ~\cite{decaria2017conservative,fiordilino2017artificial,guermond2006overview}.
	The main idea in the AC method is to replace the incompressibility condition in the NSE with an artificial compression condition. 
	Thus, the AC method decouples the velocity and pressure computations, which results in significant savings in execution time and storage.
	Furthermore, since the velocity and pressure computations are decoupled, the AC method allows the use of finite element pairs that do not satisfy the inf-sup/LBB condition~\cite{Layton08}.
	(We also note that, because the incompressibility condition is not satisfied exactly, the AC method  yields velocity fields that are not weakly divergence-free.)
	
	In this paper, we develop a novel AC-ROM that employs AC to decouple the velocity-pressure ROM approximations: 
	The fully discrete algorithm for the AC-ROM algorithm can be written as:
	\begin{subequations}\label{AC_ROM-intro}
		\begin{align}
		&\Big(\frac{u^{n+1}_R - u^{n}_R}{\Delta t}, \varphi \Big) + b^{\ast}(u_R^{n} , u^{n+1}_R ,\varphi) +\nu (\nabla u^{n+1}_R, \nabla \varphi) \label{eqn:ac-rom-velocity} \\
		&- (p^{n+1}_M , \nabla \cdot \varphi)  =( f^{n+1}, \varphi) \hspace{2.2cm} \forall \varphi \in X_R \nonumber \\
		&\varepsilon \left(\frac{p_{M}^{n+1} - p_{M}^{n}}{\Delta t}, \psi \right) + (\nabla \cdot u_R^{n+1}, \psi )= 0 \qquad \forall \psi \in Q_M, 
		\label{eqn:ac-rom-pressure}
		\end{align}
	\end{subequations}
	where $(\{ \varphi_{i}\}_{i=1}^{R},\{\psi_k\}_{k=1}^{M})$ is the ROM basis for the ROM space $(X_{R},Q_{M})$.
	The new AC-ROM~\eqref{eqn:ac-rom-velocity}--\eqref{eqn:ac-rom-pressure} has several significant advantages over the approaches (I) and (II):
	\begin{itemize} \itemsep5pt
		\item The AC-ROM {\it does not require that the velocity-pressure ROM spaces satisfy the inf-sup/LBB condition}, thus avoiding the challenges encountered in approach (I).
		\item The AC-ROM basis functions are constructed from {\it data that do not have to be weakly-divergence free}, such as those from NSE discretizations with the artificial compression, penalty, or projection methods.
		Thus, the AC-ROM avoids the challenges faced by approach (II).
	\end{itemize}
	
	\bigskip
	
	The rest of the paper is organized as follows:
	In Section~\ref{sec:notation}, we introduce some notation.
	In Section~\ref{POD_sec}, we describe the proper orthogonal decomposition, which we use to construct the ROM basis.
	In Section~\ref{sec:stability} and Section~\ref{sec:error-analysis}, we prove the stability and an error estimate of the AC-ROM~\eqref{eqn:ac-rom-velocity}--\eqref{eqn:ac-rom-pressure}, respectively.
	In Section~\ref{numex}, we investigate numerically the new AC-ROM in the simulation of a two-dimensional flow between offset cylinders.
	Finally, in Section~\ref{sec:conclusions}, we draw conclusions and outline future research directions.
	
	\section{Notation and preliminaries}\label{sec:notation}
	
	We denote by $\|\cdot\|$ and $(\cdot,\cdot)$ the $L^{2}(\Omega)$ norm and inner product, respectively, and by $\|\cdot\|_{L^{p}}$ and $\|\cdot\|_{W_{p}^{k}}$ the $L^{p}(\Omega)$ and Sobolev
	$W^{k}_{p}(\Omega)$ norms, 
	respectively. $H^{k}(\Omega)=W_{2}^{k}(\Omega)$ with
	norm $\|\cdot\|_{k}$. For a function $v(x,t)$ that is well defined on $\Omega \times [0,T]$, 
	we define the norms
	$$
	|||v|||_{2,s} : = \Big(\int_{0}^{T}\|v(\cdot,t)\|_{s}^{2}dt\Big)^{\frac{1}{2}}
	\qquad \text{and} \qquad 
	|||v|||_{\infty,s} := \text{ess\,sup}_{[0,T]}\|v(\cdot,t)\|_{s} .
	$$
	The space $H^{-1}(\Omega)$ denotes the dual space of bounded linear functionals defined on $H^{1}_{0}(\Omega)=\{v\in H^{1}(\Omega)\,:\,v=0 \mbox{ on } \partial\Omega\}$; this space is equipped with the norm
	$$
	\|f\|_{-1}=\sup_{0\neq v\in X}\frac{(f,v)}{\| \nabla v\| } 
	\quad\forall f\in H^{-1}(\Omega).
	$$
	
	The solutions spaces $X$ for the velocity and $Q$ for the pressure are respectively defined as
	$$
	\begin{aligned}
	X : =& [H^{1}_{0}(\Omega)]^{d} = \{ v \in [L^{2}(\Omega)]^{d} \,:\, \nabla v \in [L^{2}(\Omega)]^{d \times d} \ \text{and} \  v = 0 \ \text{on} \ \partial \Omega \} \\
	Q : =& L^{2}_{0}(\Omega) = \Big\{ q \in L^{2}(\Omega) \,:\, \int_{\Omega} q dx = 0 \Big\}.
	\end{aligned}
	$$
	
	A weak formulation of the NSE is given as follows: find $u:(0,T]\rightarrow X$ and $p:(0,T]\rightarrow Q$ such that, for almost all $t\in(0,T]$, satisfy 
	\begin{equation}\label{wfwf}
	\left\{\begin{aligned}
	(u_{t},v)+(u\cdot\nabla u,v)+\nu(\nabla u,\nabla v)-(p
	,\nabla\cdot v)  &  =(f,v)&\quad\forall v\in X\\
	(\nabla\cdot u,q)  &  =0&\quad\forall q\in Q\\
	u(x,0)&=u^{0}(x).&
	\end{aligned}\right.
	\end{equation}

	We denote conforming velocity and pressure finite element spaces based on a regular triangulation of $\Omega$ having maximum triangle diameter $h$ by
	$
	X_{h}\subset X$ {and} $ Q_{h}\subset Q.
	$
	We also assume that the finite element spaces satisfy the approximation properties
	$$
	\begin{aligned}
	\inf_{v_h\in X_h}\| v- v_h \|&\leq C(v) h^{s+1}&\forall v\in [H^{s+1}(\Omega)]^d\\
	\inf_{v_h\in X_h}\| \nabla ( v- v_h )\|&\leq C(v) h^s&\forall v\in [H^{s+1}(\Omega)]^d\\
	\inf_{q_h\in Q_h}\|  q- q_h \|&\leq C(q) h^s&\forall q\in H^{s}(\Omega),
	\end{aligned}
	$$
	where $C$ is a positive constant that is independent of $h$.
	


	We define the trilinear form
	$$
	b(w,u,v) = (w\cdot\nabla u,v) 
	\qquad\forall u,v,w\in [H^1(\Omega)]^d
	$$
	and the explicitly skew-symmetric trilinear form given by 
	$$
	b^{\ast}(w,u,v):=\frac{1}{2}(w\cdot\nabla u,v)-\frac{1}{2}(w\cdot\nabla v,u)
	\qquad\forall u,v,w\in [H^1(\Omega)]^d \, ,
	$$
	which satisfies the bound \cite{Layton08}
	\begin{gather}
	b^{\ast}(w,u,v)\leq C_{b^*} (\| w \| \|  \nabla w\| )^{1/2}  \| \nabla u\|  \| \nabla
	v \| \qquad\forall u, v, w \in X , \label{In1} \\
	b^{\ast}(w,u,v)\leq C_{b^*}\| \nabla w\| (\| u \| \|  \nabla u\| )^{1/2}    \| \nabla
	v \| \qquad\forall u, v, w \in X \label{In2}
	\end{gather}
	
	To ensure the uniqueness of the NSE solution and ensure that standard finite element error estimates hold, we make the following regularity assumptions on the data and true solution \cite{Layton08}:
	{
		\begin{assumption}\label{assumption:reg}
			In \eqref{wfwf} we assume that $u^0 \in X$, $f \in L^{2}(0,T;L^{2}(\Omega))$, $u \in L^{\infty}(0,T;L^{2}(\Omega))\cap L^{4}(0,T;H^{s+1}(\Omega))\cap H^{1}(0,T;H^{s+1}(\Omega))\cap
			H^{2}(0,T;L^{2}(\Omega))$, and $p \in L^{\infty}(0,T; Q \cap H^k(\Omega))$.
		\end{assumption}
	}
	We assume the following error estimate for the finite element solution of \eqref{wfwf} used to compute the velocity and pressure snapshots:
	\begin{assumption}\label{assump-FE}
		We assume that the finite element errors satisfy the following error estimates
		\begin{equation*}
		\begin{aligned}
		\sup_{n} \|u^{n} - u^{n}_{h} \|^{2} + h^{2} \|\nabla(u^{n}-u^{n}_{h})\|^{2} &\leq C(\nu,p)(h^{2s + 2} + \Delta t^{2})
		\\
		\sup_{n} \|p^{n} - p^{n}_{h} \|^{2}  &\leq C(\nu,p)(h^{2\ell} + \Delta t^{2}).
		\end{aligned}
		\end{equation*}
	\end{assumption}
	
	\begin{remark}
		Error estimates of this form have been proven for varying amounts of regularity on the continuous solution $u$ and $p$. Some examples include the incremental pressure correction schemes in \cite{G99} and chapter 7 of \cite{P97}. 
	\end{remark}
	
	The full space and time model on which we base our method is a backward Euler based artificial compression scheme with a Taylor-Hood spatial discretization, i.e., $P^{s}-P^{s-1}$ with $s \geq 2$. Given $u^{0}_h \in X_h$, $p^{0}_h \in Q_h$  for $n=0,1,2,\ldots,N-1$, find $u^{n+1}_h\in X_h$ and $p_h^{n+1}\in Q_h$ satisfying
	\begin{equation}\label{AC_FEM}
	\begin{aligned}
	&\Big(\frac{u^{n+1}_h - u^{n}_h}{\Delta t}, v_h \Big) + b^{\ast}(u_h^{n} , u^{n+1}_h ,v_h) +\nu (\nabla u^{n+1}_h, \nabla v_h) \\
	&- (p^{n+1}_h , \nabla \cdot v_h)  =( f^{n+1}, v_h) \quad \quad \qquad \forall v_h\in X_h\\
	&\varepsilon \left(\frac{p_{h}^{n+1} - p_{h}^{n}}{\Delta t},q_h \right) + (\nabla \cdot u_h^{n+1}, q_h )= 0 \qquad \forall q_h\in Q_h.
	\end{aligned}
	\end{equation}
	
	\section{Proper Orthogonal Decomposition}
	\label{POD_sec}
	In this section we briefly describe the POD method and apply it to the previously stated artificial compression algorithm. A more detailed description of this method can be found in \cite{KV01}.
	
	Given a positive integer $N$, let $0=t_0<t_1< \cdots < t_{N} = T$ denote a uniform partition of the time interval $[0,T]$.  Denote by $u_{h,S}^{n}(x)\in X_h$, $p_{h,S}^{n}(x)\in Q_h$, $n=0,\ldots,N$, the finite element solution to \eqref{AC_FEM} evaluated at $t=t_n$, $n=1,\ldots,N$. 
	
	We denote by ${u}_S^{n}$ and ${p}_S^{n}$ the vector of coefficients corresponding to the finite element functions $u_{h,S}^{n}(x)$ and $p_{h,S}^{n}$. We then define the {\em velocity snapshot matrix} $\mathbb{A}$ and {\em pressure snapshot matrix} $\mathbb{B}$ as
	\begin{equation*}
	\begin{aligned}
	\mathbb{A} = \big({u}_S^{1},{u}_S^{2}, \ldots , {u}_S^{N_{V}})\ \ \text{and} \ \ \mathbb{B} = \big({p}_S^{1},{p}_S^{2}, \ldots , {p}_S^{N_P}),
	\end{aligned}
	\end{equation*}
	%
	i.e., the columns of $\mathbb{A}$ and $\mathbb{B}$ are the finite element coefficient vectors corresponding to the discrete snapshots. The POD method then seeks a low-dimensional basis 
	\begin{equation*}
	X_R :=\text{span}\{{\varphi}_i\}_{i=1}^R  \subset X_h \ \ \text{and} \ \
	Q_M :=\text{span}\{{\psi}_i\}_{i=1}^M  \subset Q_h,
	\end{equation*}
	which can approximate the snapshot data. Let $\delta_{ij}$ denote the Kronecker delta. These bases can be determined by solving the constrained minimization problems
	\begin{equation}\label{Min-velocity}
	\begin{aligned}
	\min  \sum_{n=0}^{N} \Big \|  u_{h,s}^{n}-\sum_{j=1}^R (u_{h,s}^{n}, \varphi_j)\varphi_j\Big \| ^2 \\
	\text{subject to } (\varphi_i, \varphi_j)= \delta_{ij}\quad\mbox{for $i,j=1,\ldots,R$},
	\end{aligned}
	\end{equation}
	and 
	\begin{equation}\label{Min-pressure}
	\begin{aligned}
	\min  \sum_{n=0}^{N_{}} \Big \|  p_{h,s}^{n}-\sum_{j=1}^M (p_{h,s}^{n}, \psi_j)\psi_j\Big \| ^2 \\
	\text{subject to } (\psi_i, \psi_j)= \delta_{ij}\quad\mbox{for $i,j=1,\ldots,M$}.
	\end{aligned}
	\end{equation}
	Defining the correlation matrices  $\mathbb{C} = \mathbb{A}^{T}\mathbb{M}\mathbb{A}$ and $\mathbb{D} = \mathbb{B}^{T}\mathbb{M}\mathbb{B}$, where $\mathbb{M}$ denotes the finite element mass matrix, these problems can then be solved by considering the eigenvalue problems
	\begin{equation*}
	\mathbb{C}\vec{a}_{i} = \lambda_{i}\vec{a}_{i}.
	\end{equation*}
	and
	\begin{equation*}
	\mathbb{D}\vec{b}_{i} = \sigma_{i}\vec{b}_{i}.
	\end{equation*}
	It can then be shown the POD basis functions will be given by
	\begin{equation*}
	\vec\varphi_i = \frac{1}{\sqrt{\lambda_i}}\mathbb{A}\vec{a}_{i}, \ \ \ i = 1, \ldots, R.
	\end{equation*}
	and
	\begin{equation*}
	\vec\psi_i = \frac{1}{\sqrt{\sigma_i}}\mathbb{B}\vec{b}_{i}, \ \ \ i = 1, \ldots, M.
	\end{equation*}
	
	Using this POD basis we can now construct the AC-ROM algorithm. The construction is similar to  the full finite element approximation except we seek a solution in the POD space $(X_{R},Q_{M})$ using the basis $(\{ \varphi_{i}\}_{i=1}^{R},\{\psi_k\}_{k=1}^{M})$. The fully discrete algorithm for the AC-ROM algorithm can be written as:
	\begin{subequations}\label{AC_ROM}
		\begin{align}
		&\Big(\frac{u^{n+1}_R - u^{n}_R}{\Delta t}, \varphi \Big) + b^{\ast}(u_R^{n} , u^{n+1}_R ,\varphi) +\nu (\nabla u^{n+1}_R, \nabla \varphi) \label{AC_ROM_velocity} \\
		&- (p^{n+1}_M , \nabla \cdot \varphi)  =( f^{j,n+1}, \varphi) \hspace{2.2cm} \forall \varphi \in X_R, \nonumber\\
		&\varepsilon \left(\frac{p_{M}^{n+1} - p_{M}^{n}}{\Delta t}, \psi \right) + (\nabla \cdot u_R^{n+1}, \psi )= 0 \qquad \forall \psi \in Q_M. \label{AC_ROM_pressure}
		\end{align}
	\end{subequations}

	

	\section{Stability}
	\label{sec:stability}
	In this section we prove the unconditional, nonlinear, longtime stability of the AC-ROM  algorithm. 
	\begin{theorem}[Unconditional Stability of AC-ROM]\label{thm=stab-AC}
		For any $n$, we have the energy equality 
		\begin{gather*}
		\norm{u_R^{N+1}}^2+\epsilon \norm{p_M^{N+1}}^2+\sum_{n=0}^N\left(\norm{u_R^{n+1}-u_R^{n}}^2+\epsilon\dt\norm{p_M^{n+1}-p_M^{n}}^2\right)\\
		+2\dt\nu\sum_{n=0}^{N}\norm{\Grad{u_R^{n+1}}}^2 = \norm{u_R^0}^2+\epsilon \norm{p_M^0}+2\dt\sum_{n=0}^{N}(f^{n+1},u_R^{n+1})
		\end{gather*}
		and energy inequality
		\begin{gather*}
		\norm{u_R^{N+1}}^2+\epsilon\norm{p_M^{N+1}}^2+\sum_{n=0}^N\left(\norm{u_R^{n+1}-u_R^{n}}^2+\epsilon\norm{p_M^{n+1}-p_M^{n}}^2\right)\\
		+\dt\nu\sum_{n=0}^{N}\norm{\Grad{u_R^{n+1}}}^2 \leq \norm{u_R^0}^2+\epsilon\norm{p_M^0}+\frac{4\dt}{\nu}\sum_{n=0}^{N}\norm[-1]{f^{n+1}}^2.
		\end{gather*}
	\end{theorem}
	
	\begin{proof}
		Let $\varphi = 2\dt u_R^{n+1}$ and $\psi = 2\dt p_M^{n+1}$ in \eqref{AC_ROM}. By the polarization identity and skew-symmetry of the nonlinearity, we have
		\begin{gather*}
		\norm{u_R^{n+1}}^2-\norm{u_R^{n}}^2+\norm{u_R^{n+1}-u_R^{n}}^2+2\dt\nu\norm{\Grad{u_R^{n+1}}}^2\\
		-2\dt(p_M^{n+1},\Div{u_R^{n+1}}) = 2\dt(f^{n+1},u_R^{n+1}),\\
		\epsilon\left(\norm{p_M^{n+1}}^2-\norm{p_M^{n}}^2+\norm{p_M^{n+1}-p_M^{n}}^2\right)+2\dt(\Div{u_R^{n+1}},p_M^{n+1}) = 0.
		\end{gather*}
		Adding the two equations gives
		\begin{gather*}
		\norm{u_R^{n+1}}^2-\norm{u_R^{n}}^2+\norm{u_R^{n+1}-u_R^{n}}^2+\epsilon\left(\norm{p_M^{n+1}}^2-\norm{p_M^{n}}^2+\norm{p_M^{n+1}-p_M^{n}}^2\right)\\
		+2\dt\nu\norm{\Grad{u_R^{n+1}}}^2 = 2\dt(f^{n+1},u_R^{n+1}).
		\end{gather*}
		Summing from $n=0~\text{to}~N$ gives the energy equality above. By definition of the dual norm and Young's inequality, we have the energy inequality
		\begin{gather*}
		\norm{u_R^{N+1}}^2+\epsilon\norm{p_M^{N+1}}^2+\sum_{n=0}^N\left(\norm{u_R^{n+1}-u_R^{n}}^2+\epsilon\dt\norm{p_M^{n+1}-p_M^{n}}^2\right)\\
		+\dt\nu\sum_{n=0}^{N}\norm{\Grad{u_R^{n+1}}}^2 \leq \norm{u_R^0}^2+\epsilon\norm{p_M^0}+\frac{4\dt}{\nu}\sum_{n=0}^{N}\norm[-1]{f^{n+1}}^2,
		\end{gather*}
		proving unconditional stability.
	\end{proof}

	\section{Error Analysis}\label{sec:error-analysis}
	Next we provide an error analysis for the AC-ROM scheme. We begin by stating preliminary results.
	
	Let $\mathbb{S}_{R} = (\nabla \varphi_{i}, \nabla \varphi_{j})_{L^{2}}$ be the POD stiffness matrix and let $|||\cdot|||_{2}$ denote the matrix $2$-norm. It was shown in \cite{KV01} that this POD basis satisfies the following inverse inequality.
	\begin{lemma}[POD inverse estimate]
		\begin{equation}\label{POD:inveq}
		\|\nabla \varphi \| \leq |||\mathbb{S}_{R}|||_{2}^{1/2}\|\varphi\|, \ \ \ \forall \varphi \in X_{R}.
		\end{equation}
	\end{lemma}
	The norm $|||\mathbb{S}_{R}|||_{2}$ on the right hand side of \eqref{POD:inveq} depends on the choice of the POD basis with no universal pattern of growth with R (their number). Since $R$ is small, $|||\mathbb{S}_{R}|||_{2}$ can be precomputed giving a precise number for the right hand side of \eqref{POD:inveq}.
	
	We define the $L^{2}$ projection into the velocity space $X_{R}$, and the pressure space $Q_{M}$ as follows.  
	\begin{definition} Let $P_{R}: L^{2}(\Omega) \rightarrow X_{R}$ and $\chi_{M}: L^{2}(\Omega) \rightarrow Q_{M}$ such that
		\begin{equation}
		\begin{aligned}
		(u - P_{R}u,\varphi) &= 0, \qquad \forall \varphi \in X_{R}, \ \ \text{and} \\
		(p - \chi_{M}p, \psi) &= 0, \qquad \forall \psi \in Q_{M}.
		\end{aligned}
		\end{equation}
	\end{definition}
	
	The following lemmas, proven in \cite{KV01,singler2014new}, 
	provide bounds for the error between the snapshots and their projections onto the POD space. 
	
	\begin{lemma}\label{v-proj-errL2}[$L^{2}$ POD projection error] With $\lambda_i$ the eigenvalues of $\mathbb{C} = \mathbb{A}^{T}\mathbb{M}\mathbb{A}$, we have
		\begin{equation}
		\begin{aligned}
		\sum_{n = 0}^{N} \left \|u_{h,s}^{n} - \sum_{i=1}^{R}(u_{h,s}^{n},{\varphi_i}){\varphi}_i \right \|^{2}  &= \sum_{i=R+1}^{N_{V}} {\lambda_i}, \ \ \text{and} \\
		\sum_{n = 0}^{N} \left \| p_{h,s}^{n} - \sum_{i=1}^{M}(p_{h,s}^{n},\psi_i)\psi_i \right \|^{2} &= \sum_{i = M + 1}^{N_P} \sigma_i. 
		\end{aligned}
		\end{equation}
	\end{lemma}
	
	\begin{lemma}\label{v-proj-errH1}[$H^{1}$ POD projection error] We have
		\begin{equation}
		\begin{aligned}
		\sum_{n = 0}^{N} \left \|\nabla(u_{h,s}^{n} - \sum_{i=1}^{R}(u_{h,s}^{n},{\varphi_i}){\varphi}_i) \right \|^{2}  &=  \sum_{i=R+1}^{N_{V}} \|  \nabla {\varphi}_i\| ^2\lambda_i.
		\end{aligned}
		\end{equation}
	\end{lemma}

	The following error estimates then follow easily for the $L^{2}$ projection error into the velocity space $X_{R}$ using the techniques in \cite{KV01,singler2014new}.
	\begin{lemma}\label{pod-velo-proj-lemma}
		For any $u^{n} \in V$ the $L^{2}$ projection error into $X_{R}$  satisfies the following estimates
		\begin{equation}\label{proj-err-1}
		\begin{aligned}
		&\sup_{n} \|u^{n} - P_{R}u^{n}\|^{2} \leq C(\nu,p) \left (h^{2s  +2} + \Delta t^{2} + \sum_{i= R+1}^{N_V} {\lambda}_{i} \right), \ \ \text{and} \\
		& \sup_{n} \|\nabla(u^{n} - P_{R}u^{n})\|^{2} \leq {C(\nu,p)}\bigg(h^{2s} + |||{\mathbb S}_R|||_{2} h^{2s + 2}   + (1+|||{\mathbb S}_R|||_{2})\Delta t^2
		\\ & \hspace{4cm} + \sum_{i=R+1}^{N_{V}} \|  \nabla {\varphi}_i\| ^2\lambda_i\bigg) . 
		\end{aligned}
		\end{equation}
	\end{lemma}
	
	Similarly for the $L^{2}$ projection into the pressure space the following can be proven.
	
	\begin{lemma}\label{pod-press-proj-lemma}
		For any $p^{n} \in Q$ the $L^{2}$ projection error satisfies the following estimates
		\begin{equation}\label{proj-err-2}
		\sup_{n} \|p^{n} - \chi_{M}p^{n}\|^{2} \leq C(\nu,p) \left (h^{2\ell} + \Delta t^{2} + \sum_{i= M+1}^{N_{P}} \sigma_{i} \right).
		\end{equation}
	\end{lemma}
	Let $e_{u}$ and $e_{p}$ denote the error between the true velocity and pressure solution and their POD approximations respectively, For the error analysis we split the error for the velocity and the pressure using the $L^{2}$ projections into the spaces $X_{R}, Q_{M}$
	\begin{equation*}
	\begin{aligned}
	e^{n+1}_{u} = u^{n+1} - u^{n+1}_{R} = (u^{n+1} - P_{R}(u^{n+1})) + (P_R(u^{n+1}) - u^{n+1}_{R}) &= \eta^{n+1} - \xi_{R}^{n+1}
	\\ e^{n+1}_{p} = p^{n+1} - p^{n+1}_{M} = (p^{n+1} - \chi_{M}(p^{n+1})) + (\chi_M(p^{n+1}) - p^{n+1}_{M}) &= \kappa^{n+1} - \pi_{M}^{n+1}.
	\end{aligned}
	\end{equation*}
	
	We will see in Theorem \ref{theorem:1} that the convergence rate faces order reduction by a power of $\Delta t^{-1}$ term appearing in the error bound. This occurs due to the term $(\nabla \cdot \eta^{n+1} ,\pi_{M}^{n+1})$ arising from the continuity equation. Due to the fact the AC-ROM scheme proposed in this paper does not require the ROM velocity-pressure spaces to satisfy the $LBB_{h}$ condition, this order reduction cannot be eliminated via the usual Stokes projection. However, we will show in Theorem \ref{theorem:1} that even if the basis does not satisfy the $LBB_{h}$ condition this order reduction in the convergence rate will be improved by a multiplicative constant with size dependent upon the quality of the basis.
	
	To this end, we consider the subspace 
	\begin{equation}\label{X-div}
	X^{div}_R :=\text{span}\{\nabla \cdot \varphi_i\}_{i=1}^R  \subset L^{2}(\Omega),
	\end{equation}
	and recall from \cite{EV91} the strengthened Cauchy-Buniakowskii-Schwarz (CBS) inequality commonly used in the analysis of multilevel methods \cite{AG83,EV91}.
	
	\begin{lemma}\label{CBS}
		Given a Hilbert space V and two finite dimensional subspaces  $V_{1} \subset V$ and $V_{2} \subset V$ with trivial intersection: 
		\begin{equation*}
		V_{1} \cap V_{2} = \{0\},
		\end{equation*}
		then there exists $ 0 \leq \alpha < 1$ such that
		\begin{equation*}
		|(v_{1},v_{2})| \leq \alpha \|v_{1}\|\|v_{2}\| \ \ \ \forall v_1 \in V_{1}, v_2 \in V_{2}.
		\end{equation*} 
	\end{lemma}

	Considering $X^{div}_R$ and $Q_{M}$, we are interested in computing the exact constant  $\alpha$ between these spaces. This is equivalent to finding
	the first principal angle defined as 
	\begin{equation}\label{first_principal_angle}
	\theta_{1} := \min \left\{\arccos\left({\frac{|(v,\psi)|}{\|v\|\|\psi\|}}\right) \bigg| v \in X^{div}_{R}, \psi \in Q_{M} \right\},
	\end{equation} 
	with $0\leq \theta_{1} \leq \frac{\pi}{2}$. 
	
	The problem of computing angles between subspaces was introduced by Jordan in 1875 \cite{J1875} and studied by Friedrichs in 1937 \cite{F1937}. Recently, principal angles were  used to improve the accuracy of  reduced basis schemes for optimization problems in \cite{M14}. They can be calculated using either QR factorization or SVD of the orthogonal bases of the spaces in Lemma \ref{CBS}, as outlined in \cite{WS03}. More efficient and stable schemes for calculating principal angles were also developed in \cite{KA02}. We note that due to the relatively small size of the pressure and velocity reduced basis, the QR or SVD approach is sufficient in this setting. This procedure will be briefly outlined in section 6.
	
	Using the strengthened CBS inequality, we get the following bound on the error term arising from the continuity equation.
	
	\begin{lemma}\label{lemma:ortho_bd}
		Let $u^{n+1} = u(t^{n+1},x)$ be the exact solution of the NSE and let $\eta^{n+1} = u^{n+1} - P_{R}(u^{n+1})$ denote the projection error. Defining $\alpha = \cos(\theta_{1})$, where $\theta_{1}$ is given in \eqref{first_principal_angle} the following bound holds
		\begin{equation*}
		|(\nabla \cdot \eta^{n+1},\psi)| \leq \alpha \|\nabla \cdot  \eta^{n+1}\| \| \psi \| \qquad \forall \psi \in Q_{M}.
		\end{equation*}
	\end{lemma}
	\begin{proof}
		Since $u^{n+1}$ is the exact solution to the NSE it follows that $\nabla \cdot u^{n+1} = 0$ and therefore $(\nabla \cdot u^{n+1},\psi) = 0 \ \forall \psi \in Q_{M}$. This gives
		\begin{equation*}
		|(\nabla \cdot \eta^{n+1},\psi)| = |(\nabla \cdot u^{n+1} - \nabla \cdot P_{R}(u^{n+1}) ,\psi)| = |(\nabla \cdot  P_{R}u^{n+1},\psi)| \ \forall \psi \in Q_{M}.
		\end{equation*}
		It then follows from the fact that  $\nabla \cdot  P_{R}u^{n+1} \in X_{R}^{div}$, Lemma \ref{CBS}, and $\nabla \cdot u^{n+1} = 0$.
		\begin{equation*}
		|(\nabla \cdot  P_{R}u^{n+1},\psi)| \leq \alpha \|\nabla \cdot  P_{R}u^{n+1}\| \| \psi \| =  \alpha \|\nabla \cdot  \eta^{n+1}\| \| \psi \| \ \forall \psi \in Q_{M}.
		\end{equation*}
	\end{proof}
	
	
	We are now ready to state the full error estimate.

	\begin{theorem}\label{theorem:1}
		Consider AC-ROM \eqref{AC_ROM} and the partition $0 = t_{0} < t_{1} < \cdots < t_{N} = T$ used in section \ref{POD_sec}. Letting $\epsilon = \mathcal{O}(\dt)$ it then holds
		\begin{gather*}
		\norm{e^{N+1}_u}^2+\epsilon\norm{e^{N+1}_p}^2+ \frac{\nu\dt}{2}\norm{\Grad{e^{N+1}_u}}^2+\frac{\nu}{2}\norm[2,0]{\Grad{e_u}}^2
		\\
		+\sum_{n=0}^{N}\left(\norm{e^{n+1}_u-e^{n}_u}^2+\frac{\epsilon}{2}\norm{e^{n+1}_p-e^{n}_p}^2\right)
		\\
		\leq \norm{\eta^{N+1}}^2+\epsilon\norm{\kappa^{N+1}}^2+ \frac{\nu\dt}{2}\norm{\Grad{\eta^{N+1}}}^2+\frac{\nu}{2}\norm[2,0]{\Grad{\eta}}^2
		\\
		+\sum_{n=0}^{N}\left(\norm{\eta^{n+1}-\eta^{n}}^2+\frac{\epsilon}{2}\norm{\kappa^{n+1}-\kappa^{n}}^2\right)
		\\ +C\exp\left({\frac{\widetilde{C}T}{\nu^{3}}}\right)\bigg(\norm{\xi^{0}_R}^2+\epsilon\norm{\pi^{0}_M}^2+\dt\nu\norm{\Grad{\xi^{0}_R}}^2
		\\
		+\left(\frac{1}{\nu}+\alpha^{2}\dt^{-1}\right)\norm[2,0]{\Grad{\eta}}^2+\frac{1}{\nu}\norm[2,0]{\kappa}^2+\frac{\dt^2}{\nu}\norm[L^2(0,T;L^2(\Omega))]{u_{tt}}^2
		\\
		+\frac{\dt^2}{\nu}\norm[L^2(0,T;L^2(\Omega))]{\Grad{u_t}}^2 + \frac{\dt^{1/2}}{\nu^{3/2}}\left(\sum_{n=0}^{N} \| \nabla \eta^{n+1} \|^{4}\right)^{\frac{1}{2}}
		\\
		+\dt^3\norm[L^2(0,T;L^2(\Omega))]{p_t}^2+\dt\norm[\infty]{p_t}^2\bigg).
		\end{gather*}
		
	\end{theorem}
	\begin{proof}
		The weak solution of the NSE satisfies
		\begin{equation}
		\begin{aligned}
		\label{er-eq1-lc}
		\left(\frac{u^{n+1} - u^n}{\Delta t}, \varphi \right) + &\bstar{u^{n+1}}{u^{n+1}}{\varphi} + \nu(\nabla u^{n+1},\Grad{\varphi})\\
		&- (p^{n+1}, \nabla \cdot \varphi) = (f^{n+1},\varphi) + \tau_u(u^{n+1};\varphi)
		\end{aligned}
		\end{equation}
		\begin{equation}
		\label{er-eq2-lc}
		\varepsilon \left(\frac{p^{n+1} - p^{n}}{\Delta t} \right) + (\nabla \cdot u^{n+1},\psi) = \tau_p(p^{n+1};\psi),
		\end{equation}
		where
		\begin{equation}
		\begin{aligned}
		&\tau_u(u^{n+1};\varphi) = \left(\frac{u^{n+1} - u^n}{\Delta t} - u_t(t^{n+1}),\varphi \right) \\
		&\tau_p(p^{n+1};\psi) = \left(\frac{1}{\dt}\int_{t^n}^{t^{n+1}}p_t(t)dt,\psi \right).
		\end{aligned}
		\end{equation}
		Now subtracting \eqref{AC_ROM_velocity} from \eqref{er-eq1-lc} and  \eqref{AC_ROM_pressure} from \eqref{er-eq2-lc} we have
		\begin{equation}\label{er-eq3-lc}
		\begin{aligned}
		&\left(\frac{\xi^{n+1}_{R} -\xi_{R}^{n}}{\Delta t},\varphi \right) + \nu(\nabla \xi^{n+1}_{R},\nabla \varphi) - (\pi_{M}^{n+1},\nabla \cdot \varphi)
		\\
		& = \left(\frac{\eta^{n+1} -\eta^{n}}{\Delta t},\varphi \right) + \nu(\nabla \eta^{n+1},\nabla \varphi) - (\kappa^{n+1},\nabla \cdot \varphi)
		\\
		&  +\bstar{u^{n+1}}{u^{n+1}}{\varphi} -\bstar{u^{n}_R}{u^{n+1}_R}{\varphi} - \tau_u(u^{n+1};\varphi)
		\end{aligned}
		\end{equation}
		and
		\begin{equation}\label{er-eq4-lc}
		\begin{aligned}
		&\varepsilon \left( \frac{\pi^{n+1}_{M} - \pi_{M}^{n}}{\Delta t},\psi \right) + (\nabla \cdot \xi^{n+1}_{R},\psi )
		\\ 
		&= \varepsilon \left( \frac{\kappa^{n+1} - \kappa^{n}}{\Delta t} ,\psi \right) + (\nabla \cdot \eta^{n +1},\psi) - \tau_{p}(p^{n+1};\psi).
		\end{aligned}
		\end{equation}
		Setting $\varphi = 2 \Delta t \xi_{R}^{n+1}$ and $\psi = 2 \Delta t \pi_{M}^{n+1}$, we use the fact that $\left( \frac{\eta^{n+1} - \eta^{n}}{\Delta t} ,\xi_{R}^{n+1} \right) = 0$  and $\left( \frac{\kappa^{n+1} - \kappa^{n}}{\Delta t} ,\pi_{M}^{n+1} \right) = 0$ by the definition of the $L^2$ projection. Adding \eqref{er-eq3-lc} to \eqref{er-eq4-lc} and using the polarization identity yields
		\begin{gather}\label{eq:514}
		(\|\xi_{R}^{n+1} \|^{2} + \epsilon \| \pi_{M}^{n+1} \|^{2}) - (\|\xi_{R}^{n} \|^{2} + \epsilon \| \pi_{M}^{n} \|^{2}) + \|\xi_{R}^{n+1} - \xi_{R}^{n}  \|^{2} 
		\\
		\nonumber + \epsilon \|\pi_{M}^{n+1} - \pi_{M}^{n} \|^{2} + 2 \Delta t \nu \|\nabla \xi^{n+1}_R\|^{2} = 2 \Delta t \nu (\nabla \eta^{n+1},\nabla \xi^{n+1}_{R})- 2 \Delta t(\kappa^{n+1}, \nabla \cdot \xi_{R}^{n+1})
		\\
		\nonumber + 2 \Delta t (\nabla \cdot \eta^{n+1},\pi^{n+1}_{M}) + 2 \Delta t b^{*}(u^{n},u^{n+1}, \xi^{n+1}_{R}) - 2 \Delta t b^{*}(u_{R}^{n},u_{R}^{n+1}, \xi^{n+1}_{R})
		\\
		\nonumber - 2 \Delta t \tau_{u}(u^{n+1};\xi^{n+1}_R) - 2 \Delta t \tau_{p}(p^{n+1};\pi^{n+1}_M).
		\end{gather}
		By Poincar\'{e} and Young's inequality we bound the first two terms on the right hand side of \eqref{eq:514}
		\begin{equation}
		\begin{aligned}
		2 \Delta t \nu (\nabla \eta^{n+1},\nabla \xi^{n+1}_{R}) &\leq \frac{\Delta t \nu}{\delta_{1}} \|\nabla \eta^{n+1}\|^{2} + \delta_{1}\Delta t \nu \|\nabla \xi_{R}^{n+1}\|^{2} 
		\\
		- 2 \Delta t(\kappa^{n+1}, \nabla \cdot \xi_{R}^{n+1}) &\leq \frac{\Delta t}{\nu \delta_{2}}\|\kappa^{n+1}\|^{2} + \delta_{2} \Delta t \nu \|\nabla \xi_{R}^{n+1}\|^{2}.
		\end{aligned}
		\end{equation}
		For the third term on the right of \eqref{eq:514}, adding and subtracting $2 \Delta t (\nabla \cdot \eta^{n+1} ,\pi_{M}^{n})$, applying Young's inequality, and Lemma \ref{lemma:ortho_bd} yields
		\begin{gather*}
		2 \Delta t (\nabla \cdot \eta^{n+1} ,\pi_{M}^{n+1}) =  2 \Delta t \left((\nabla \cdot \eta^{n+1} ,\pi_{M}^{n+1} - \pi_{M}^{n}) + (\nabla \cdot \eta^{n+1} ,\pi_{M}^{n})\right)
		\\
		\leq 2 \dt \alpha \| \nabla \eta^{n+1}\| \|\pi_{M}^{n+1} - \pi_{M}^{n}\| + 2 \dt \alpha \|\nabla \cdot \eta^{n+1} \| \|\pi_{M}^{n} \|
		\\
		\leq \frac{\alpha^{2}\Delta t^{2}}{\delta_{3}\epsilon}\|\nabla \eta^{n+1}\|^{2} + \delta_{3}\epsilon \|\pi_{M}^{n+1} - \pi_{M}^{n}\|^{2} + \frac{\Delta t \alpha^{2} }{\epsilon \delta_{4}}\|\nabla \eta^{n+1}\|^{2} + \delta_{4} \epsilon \Delta t\|\pi_{M}^{n}\|^{2}
		\\
		= \frac{ \alpha^{2}(\delta_4 \Delta t^{2} + \delta_{3}\Delta t)}{\epsilon \delta_3 \delta_4}\|\nabla \eta^{n+1}\|^{2} + \delta_{3}\epsilon \|\pi_{M}^{n+1} - \pi_{M}^{n}\|^{2} + \delta_{4} \epsilon \Delta t\|\pi_{M}^{n}\|^{2}.
		\end{gather*}
		
		\noindent Next, for the nonlinear terms we add and subtract $\bstar{u_{R}^{n}}{u^{n+1}}{\xi^{n+1}_R}$ and \newline 
		$\bstar{u^{n}}{u^{n+1}}{\xi^{n+1}_R}$. This yields, by skew-symmetry,
		\begin{gather*}
		2\Delta t \bstar{u^{n+1}}{u^{n+1}}{\xi^{n+1}_R} - 2\Delta t \bstar{u_{R}^{n}}{u_{R}^{n+1}}{\xi^{n+1}_{R}}
		\\ 
		= 2\dt\bstar{u^{n+1}-u^{n}}{u^{n+1}}{\xi^{n+1}_R}+ 2\Delta t \bstar{u_{R}^{n}}{e_{u}^{n+1}}{\xi^{n+1}_{R}} + 2\Delta t \bstar{e_{u}^{n}}{u^{n+1}}{\xi^{n+1}_{R}}
		\\
		= 2\dt\bstar{u^{n+1}-u^{n}}{u^{n+1}}{\xi^{n+1}_R}- 2\Delta t \bstar{\xi_{R}^{n}}{u^{n+1}}{\xi^{n+1}_{R}} + 2\Delta t\bstar{\eta^{n}}{u^{n+1}}{\xi^{n+1}_{R}} \\
		+2\Delta t \bstar{u_{R}^{n}}{\eta^{n+1}}{\xi^{n+1}_{R}}.
		\end{gather*}
		The first three nonlinear terms are now bounded using the Sobolev Imbedding Theorem, Young's inequality, \eqref{In1}, and \eqref{In2}
		\begin{gather}
		2\dt\bstar{u^{n+1}-u^{n}}{u^{n+1}}{\xi^{n+1}_R}\leq \frac{C\dt^2}{\delta_{5}\nu}\norm{\Grad{u^{n+1}}}^{2}\norm[L^2(t^{n},t^{n+1};L^2(\Omega))]{\Grad{u_t}}^2\\
		+\delta_{5}\dt\nu\norm{\Grad{\xi^{n+1}_R}}^2\nonumber\\
		2\Delta t b^{*}(u_{R}^{n},\eta^{n+1}, \xi^{n+1}_{R}) \leq \frac{C \Delta t}{\delta_{6}\nu}\| \nabla u_{R}^{n} \| \| u_{R}^{n} \| \| \nabla \eta^{n+1} \|^{2} + \delta_{6} \Delta t \nu \|\nabla \xi^{n+1}_{R}\|^{2}
		\\
		2\Delta t b^{*}(\eta^{n},u^{n+1}, \xi^{n+1}_{R}) \leq \frac{C \Delta t}{\delta_{7}\nu}\| \nabla u^{n+1} \|^{2} \| \nabla \eta^{n} \|^{2} + \delta_{7} \Delta t \nu \|\nabla \xi^{n+1}_{R}\|^{2}
		\\
		- 2\Delta t b^{*}(\xi_{R}^{n},u^{n+1}, \xi^{n+1}_{R}) \leq \frac{C \Delta t}{\delta_{8}^{2} \delta_{9}\nu^{3}}\|\nabla u^{n+1} \|^{4}\|\xi^{n}_{R} \|^{2} + \delta_{8} \Delta t \nu \|\nabla \xi^{n+1}_{R} \|^{2} \\
		+ \delta_{9} \Delta t \nu \|\nabla \xi^{n}_{R} \|^{2} \nonumber. 
		\end{gather}
		Dealing with the consistency terms, by Taylor's Theorem and Young's inequality we have
		\begin{equation}
		\begin{aligned}
		-2 \Delta t \tau_u(u^{n+1};\xi^{n+1}_R) &\leq \|\frac{u^{n+1} - u^{n}}{\Delta t} - u_{t}(t^{n+1})\|\|\xi_{R}^{n+1}\|
		\\
		& \leq \frac{C \Delta t^{2}}{\nu \delta_{10}}\| u_{tt}\|^{2}_{L^{2}(t^{n},t^{n+1}; L^{2})} + \nu \Delta t \delta_{10} \|\Grad{\xi_{R}^{n+1}}\|^{2}
		\end{aligned}
		\end{equation}
		and, by adding and subtracting by $2\dt\tau_{p}(p^{n+1};\pi^{n}_M)$, we have
		\begin{gather}
		- 2 \Delta t \tau_{p}(p^{n+1};\pi^{n+1}_M)\leq \frac{C\epsilon\dt^2}{\delta_{11}}\norm[L^2(t^{n},t^{n+1};L^2(\Omega))]{p_t}^2+\frac{C\epsilon\dt}{\delta_{12}}\norm[\infty]{p_t}^2\\\nonumber
		+ \delta_{11}\epsilon\norm{\pi^{n+1}_M-\pi^{n}_M}^2+\delta_{12}\epsilon\dt\norm{\pi^{n}_M}^2.
		\end{gather}
		Letting $\delta_1=\delta_2=\delta_5=\delta_6=\delta_7=\delta_8=\delta_{10}=\frac{1}{14}$, $\delta_3=\delta_{11}=\frac{1}{4}$, $\delta_4=\delta_{12}=\frac{1}{2}$, $\delta_9 = 1$, and rearranging/combining terms we have
		\begin{gather*}
		(\|\xi_{R}^{n+1} \|^{2} + \epsilon \| \pi_{M}^{n+1} \|^{2}) - (\|\xi_{R}^{n} \|^{2} + \epsilon \| \pi_{M}^{n} \|^{2}) + \|\xi_{R}^{n+1} - \xi_{R}^{n}  \|^{2} 
		\\
		\nonumber + \frac{\epsilon}{2} \|\pi_{M}^{n+1} - \pi_{M}^{n} \|^{2} + \frac{\nu \Delta t}{2} \|\nabla \xi^{n+1}_R\|^{2} + \frac{\nu \Delta t}{2} (\|\nabla \xi^{n+1}_R\|^{2} - \|\nabla \xi^{n}_{R} \|^{2}) \leq  \epsilon\dt\norm{\pi^{n}_M}^2 
		\\
		+ \dt \nu \|\nabla \eta^{n+1}\|^{2}
		+ \frac{\dt\alpha^{2}}{\epsilon}\|\nabla \eta^{n+1}\|^{2} + \frac{C\dt^2}{\nu}\norm{\Grad{u^{n+1}}}\norm[L^2(t^{n},t^{n+1};L^2(\Omega))]{\Grad{u_t}}^2 \\
		+\frac{\dt}{\nu}\|\kappa^{n+1}\|^{2}
		+ \frac{\dt^{2}}{\epsilon}\|\nabla \eta^{n+1}\|^{2} + \frac{C \Delta t}{\nu}\| u_{R}^{n} \|\| \nabla u_{R}^{n} \| \| \nabla \eta^{n+1} \|^{2} 
		\\ 
		+ \frac{C \Delta t}{\nu}\| \nabla u^{n+1} \|^{2} \| \nabla \eta^{n} \|^{2} 
		+ \frac{C \Delta t}{\nu^{3}}\|\nabla u^{n+1} \|^{4}\|\xi^{n}_{R} \|^{2} +
		\frac{C \Delta t^{2}}{\nu}\| u_{tt}\|^{2}_{L^{2}(t^{n},t^{n+1}; L^{2})}  
		\\
		+{C\epsilon\dt^2}\norm[L^2(t^{n},t^{n+1};L^2(\Omega))]{p_t}^2  + {C\epsilon\dt}\norm[\infty]{p_t}^2.
		\end{gather*}
		We note by Theorem \ref{thm=stab-AC} and the Cauchy Schwartz inequality that it follows
		\begin{equation}\label{nonlinear-stab-bound}
		\begin{aligned}
		\dt&\sum_{n=0}^{N}\| u_{R}^{n} \|\| \nabla u_{R}^{n} \| \| \nabla \eta^{n+1} \|^{2} \leq \max_{n=0,\ldots,N}\|u_{R}^{n}\|\dt\sum_{n=0}^{N}\| \nabla u_{R}^{n} \|\| \nabla \eta^{n+1} \|^{2}
		\\
		& \leq \max_{n=0,\ldots,N}\|u_{R}^{n}\|\left(\dt\sum_{n=0}^{N} \| \nabla u_{R}^{n} \|^{2}\right)^{\frac{1}{2}}\left(\dt\sum_{n=0}^{N} \| \nabla \eta^{n+1} \|^{4}\right)^{\frac{1}{2}}
		\\
		&\leq \frac{C_{stab}\dt^{1/2}}{\nu^{1/2}}\left(\dt\sum_{n=0}^{N} \| \nabla \eta^{n+1} \|^{4}\right)^{\frac{1}{2}}. 
		\end{aligned}
		\end{equation}
		By Theorem \ref{thm=stab-AC}, Assumption \ref{assumption:reg}, \eqref{nonlinear-stab-bound}, the fact that $\epsilon = \mathcal{O}(\dt)$, combining all inequalities, taking a maximum $\widetilde{C}$ over all constants, and summing from $n=0$ to $N$ yields
		\begin{gather*}
		\norm{\xi^{N+1}_R}^2+ \epsilon\norm{\pi^{N+1}_M}^2+\frac{\dt\nu}{2}\norm{\Grad{\xi^{N+1}_R}}^2+\frac{\nu}{2}\norm[2,0]{\Grad{\xi_R}}^2
		\\
		+\sum_{n=0}^{N}\left(\norm{\xi^{n+1}_R-\xi^{n}_R}^2+\frac{\epsilon}{2}\norm{\pi^{n+1}_M-\pi^{n}_M}^2\right)\leq \norm{\xi^{0}_R}^2+\epsilon\norm{\pi^{0}_M}^2
		\\
		+\dt\nu\norm{\Grad{\xi^{0}_R}}^2+\frac{\widetilde{C}\dt}{\nu^{3}}\sum_{n=0}^{N}\left(\norm{\xi^{n}_R}^2+\epsilon\norm{\pi^{n}_M}^2\right)
		\\
		+\widetilde{C}\bigg[\left(\frac{1}{\nu}+\alpha^{2}\dt^{-1}\right)\norm[2,0]{\Grad{\eta}}^2+\frac{1}{\nu}\norm[2,0]{\kappa}^2+
		\\
		\frac{\dt^{1/2}}{\nu^{3/2}}\left(\sum_{n=0}^{N} \| \nabla \eta^{n+1} \|^{4}\right)^{\frac{1}{2}}+\frac{\dt^2}{\nu}\norm[L^2(0,T;L^2(\Omega))]{u_{tt}}^2
		\\
		+\frac{\dt^2}{\nu}\norm[L^2(0,T;L^2(\Omega))]{\Grad{u_t}}^2+\dt^3\norm[L^2(0,T;L^2(\Omega))]{p_t}^2+\dt\norm[\infty]{p_t}^2\bigg].
		\end{gather*}
		Therefore, by a discrete Gr\"{o}nwall inequality,
		\begin{gather*}
		\norm{\xi^{N+1}_R}^2+\epsilon\norm{\pi^{N+1}_M}^2+\frac{\dt\nu}{2}\norm{\Grad{\xi^{N+1}_R}}^2+\frac{\nu}{2}\norm[2,0]{\Grad{\xi_R}}^2\\
		+\sum_{n=0}^{N}\left(\norm{\xi^{n+1}_R-\xi^{n}_R}^2+\frac{\epsilon}{2}\norm{\pi^{n+1}_M-\pi^{n}_M}^2\right)\\
		\leq C\exp\left({\frac{\widetilde{C}T}{\nu^{3}}}\right)\bigg(\norm{\xi^{0}_R}^2+\epsilon\norm{\pi^{0}_M}^2+\dt\nu\norm{\Grad{\xi^{0}_R}}^2\\
		+\left(\frac{1}{\nu}+\alpha^{2}\dt^{-1}\right)\norm[2,0]{\Grad{\eta}}^2+\frac{1}{\nu}\norm[2,0]{\kappa}^2+\frac{\dt^2}{\nu}\norm[L^2(0,T;L^2(\Omega))]{u_{tt}}^2\\
		+\frac{\dt^2}{\nu}\norm[L^2(0,T;L^2(\Omega))]{\Grad{u_t}}^2 + \frac{\dt^{1/2}}{\nu^{3/2}}\left(\sum_{n=0}^{N} \| \nabla \eta^{n+1} \|^{4}\right)^{\frac{1}{2}}
		\\
		+\dt^3\norm[L^2(0,T;L^2(\Omega))]{p_t}^2+\dt\norm[\infty]{p_t}^2\bigg).
		\end{gather*}
		By the triangle inequality we have $\|e_{u}^{n+1}\|^{2} \leq 2 (\|\eta^{n+1}\|^{2} + \| \xi_{R}^{n+1}\|^{2})$, as well as \newline $\|e_{p}^{n+1}\|^{2} \leq 2 (\|\kappa^{n+1}\|^{2} + \| \pi_{M}^{n+1}\|^{2})$. Applying this and and taking a maximum among constants the result follows.
	\end{proof}

	We see in Theorem \ref{theorem:1} that whether or not $\dt^{-1}$ order reduction occurs depends upon the constant $\alpha$. If $\alpha^{2} << \dt$ the order reduction will be alleviated. On the other hand for $\alpha^{2} >> \dt$ there will be no improvement in the convergence. 
	
	In \cite{BN85,BN84} it was determined for a $P^{1}-P^{0}$ finite element pair that $\alpha = \sqrt{\frac{3}{2}}h$. To our knowledge no similar calculations have been done for other finite element pairs such as Taylor-Hood.  It is unclear what impact the choice of offline basis functions has on the principal angle for the POD spaces. Additionally the POD basis functions will be problem dependent. It is possible that certain problem settings yield large principal angles independent of the choice of finite element basis functions used in the offline stage.

	%
	
	\section{Numerical Experiments}\label{numex}
	In this section, we perform a numerical investigation of the new AC-ROM algorithm \eqref{AC_ROM}.  First, we show that the AC-ROM algorithm yields accurate velocity and pressure approximations without enforcing the LBB condition or requiring weakly divergence free snapshots.  Then, we illustrate numerically the theoretical scalings proved in Theorem \ref{theorem:1}. In particular, we show that the AC-ROM algorithm yields first order scalings with respect to the time step, $\Delta t$. All computations are done using the FEniCS software suite \cite{LNW12} and all meshes generated via the built in meshing package \textbf{mshr}.
	
	\subsection{Problem Setting}\label{sec:numerical-experiments}
	For the numerical experiments we consider the two-dimensional flow between offset cylinders used in \cite{GJS17,JL14}. The domain is a disk with a smaller off-center disc inside. Let $r_{1}=1$, $r_{2}=0.1$, $c_{1}=1/2$, and $c_{2}=0$; then, the domain is given by
	\[
	\Omega=\{(x,y):x^{2}+y^{2}\leq r_{1}^{2} \text{ and } (x-c_{1})^{2}%
	+(y-c_{2})^{2}\geq r_{2}^{2}\}.
	\]

	\begin{figure}[!ht]
		\centering
		\begin{subfigure}{0.28\linewidth}
			\centering
			\includegraphics[width = 1\linewidth]{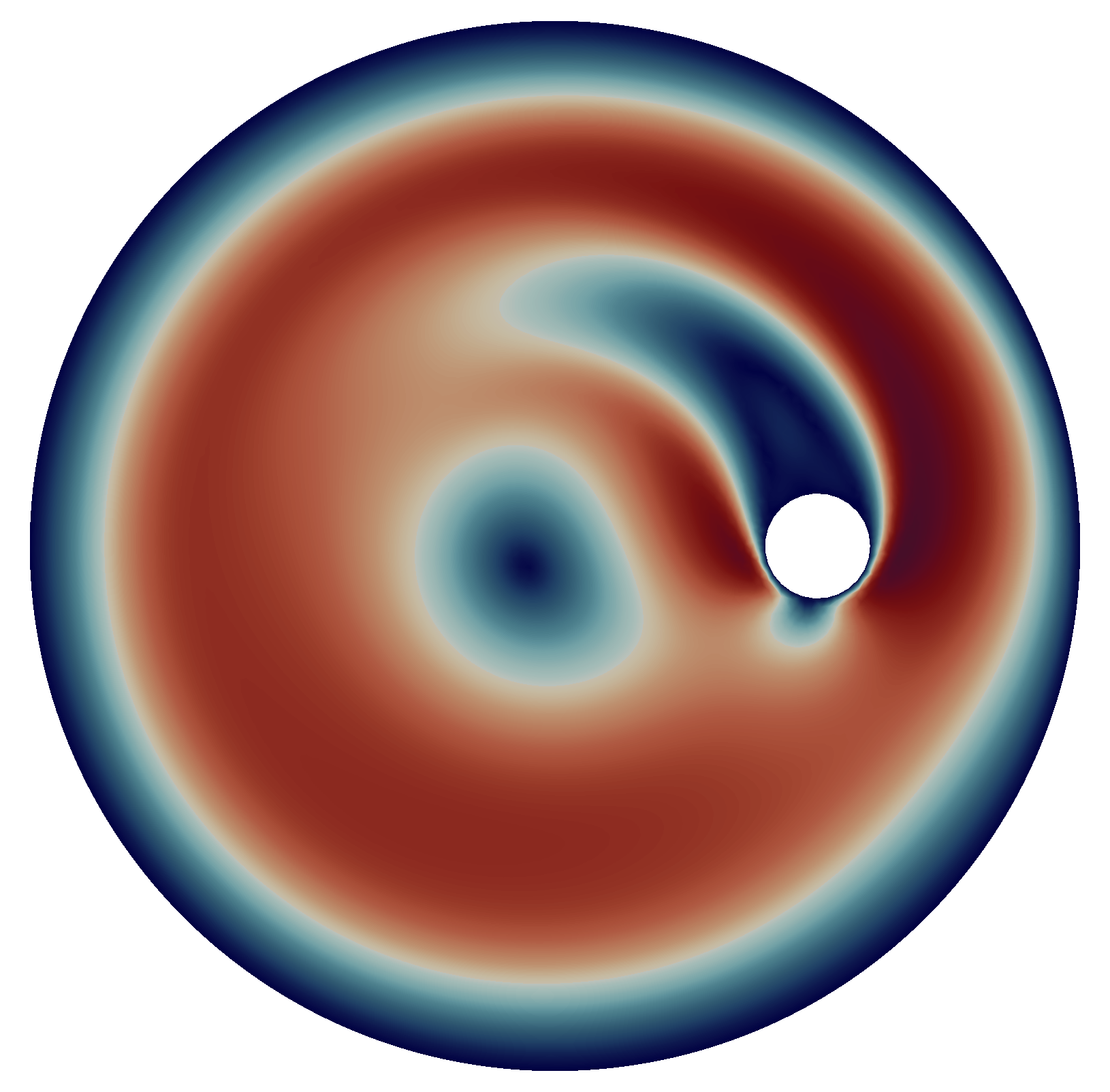}
		\end{subfigure}
		\begin{subfigure}{0.28\linewidth}
			\centering
			\includegraphics[width = 1\linewidth]{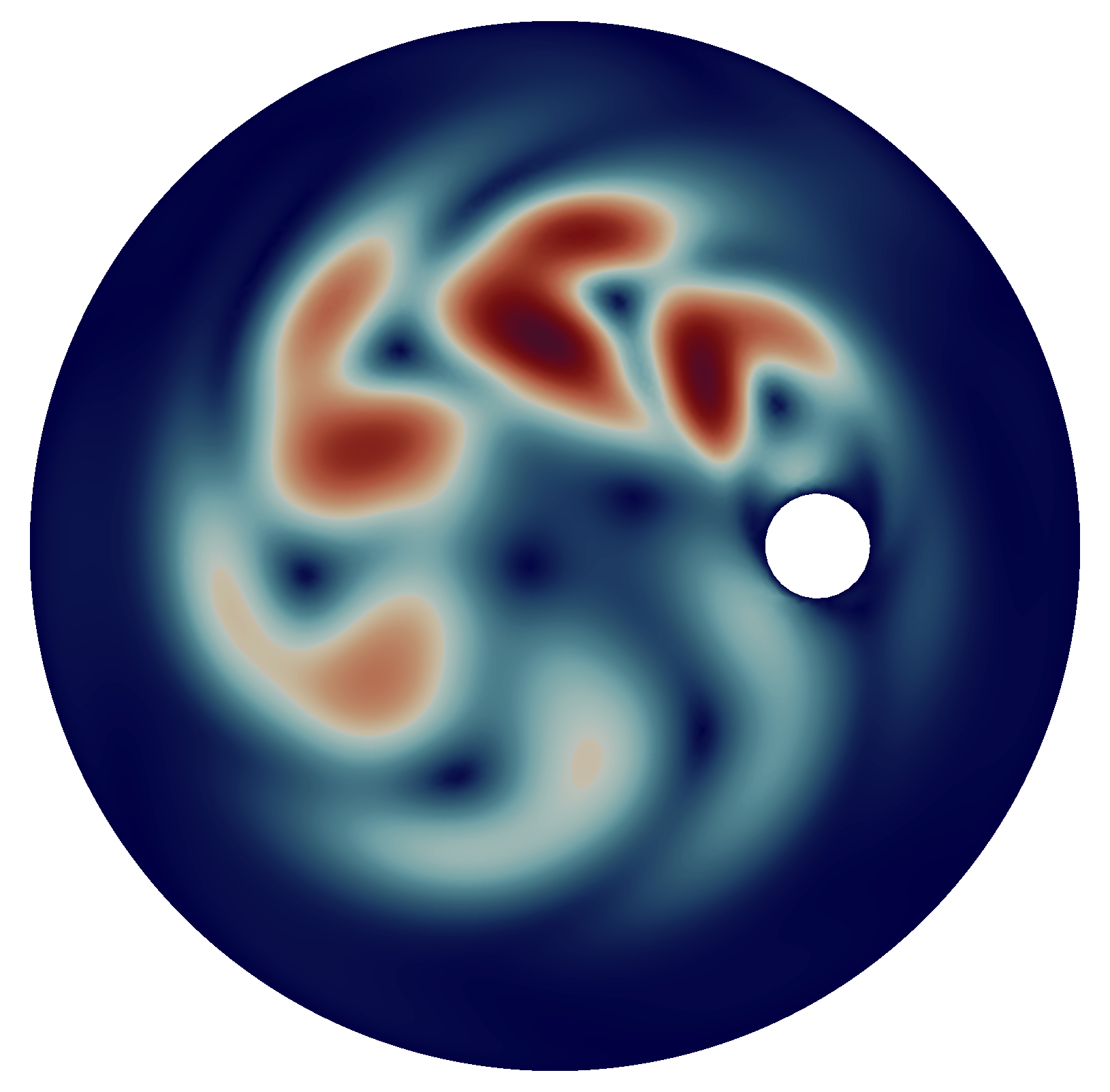}
		\end{subfigure}
		\begin{subfigure}{0.28\linewidth}
			\centering
			\includegraphics[width = 1\linewidth]{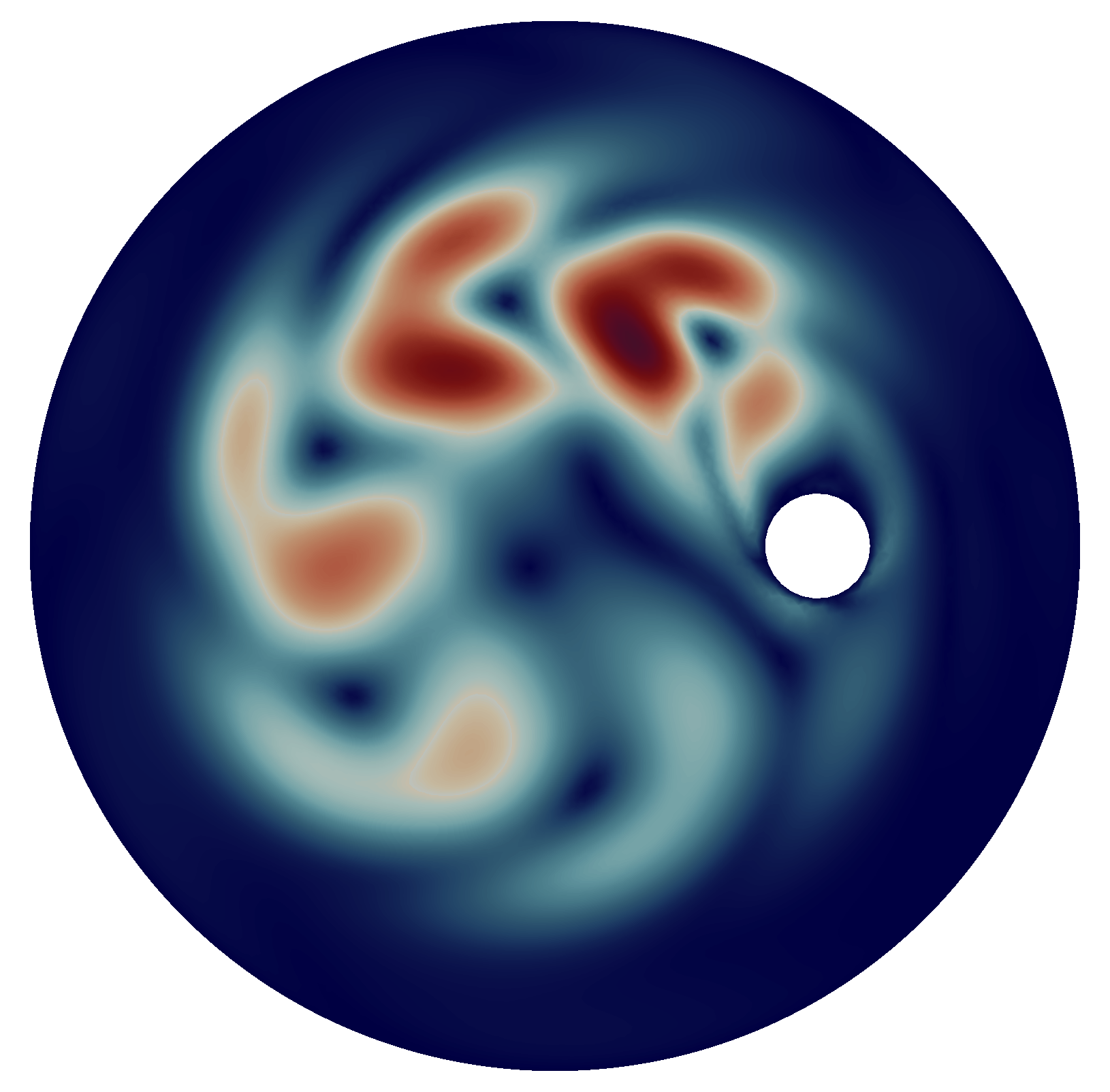}
		\end{subfigure}
		\begin{subfigure}{0.28\linewidth}
			\centering
			\includegraphics[width = 1\linewidth]{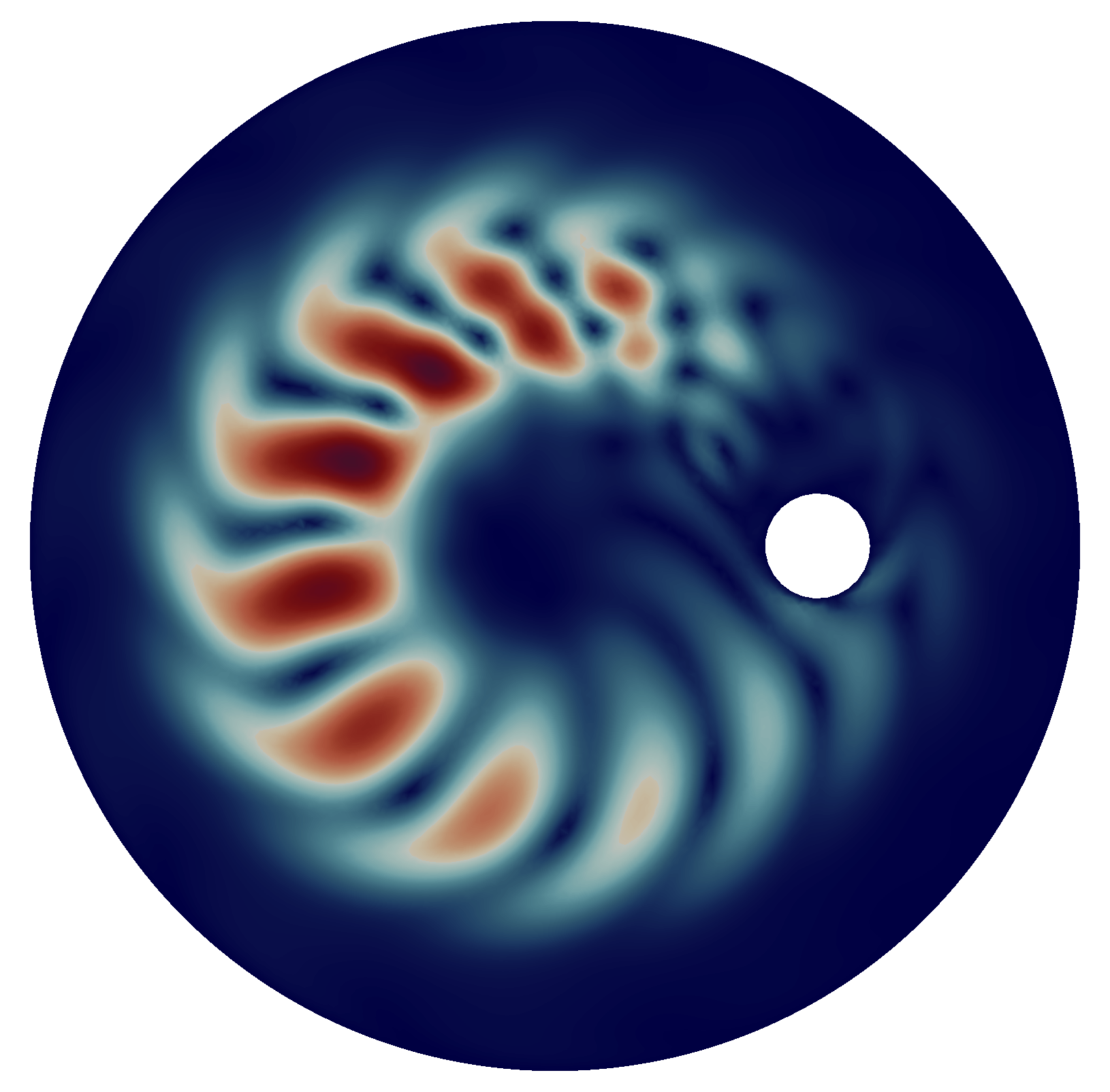}
		\end{subfigure}
		\begin{subfigure}{0.28\linewidth}
			\centering
			\includegraphics[width = 1\linewidth]{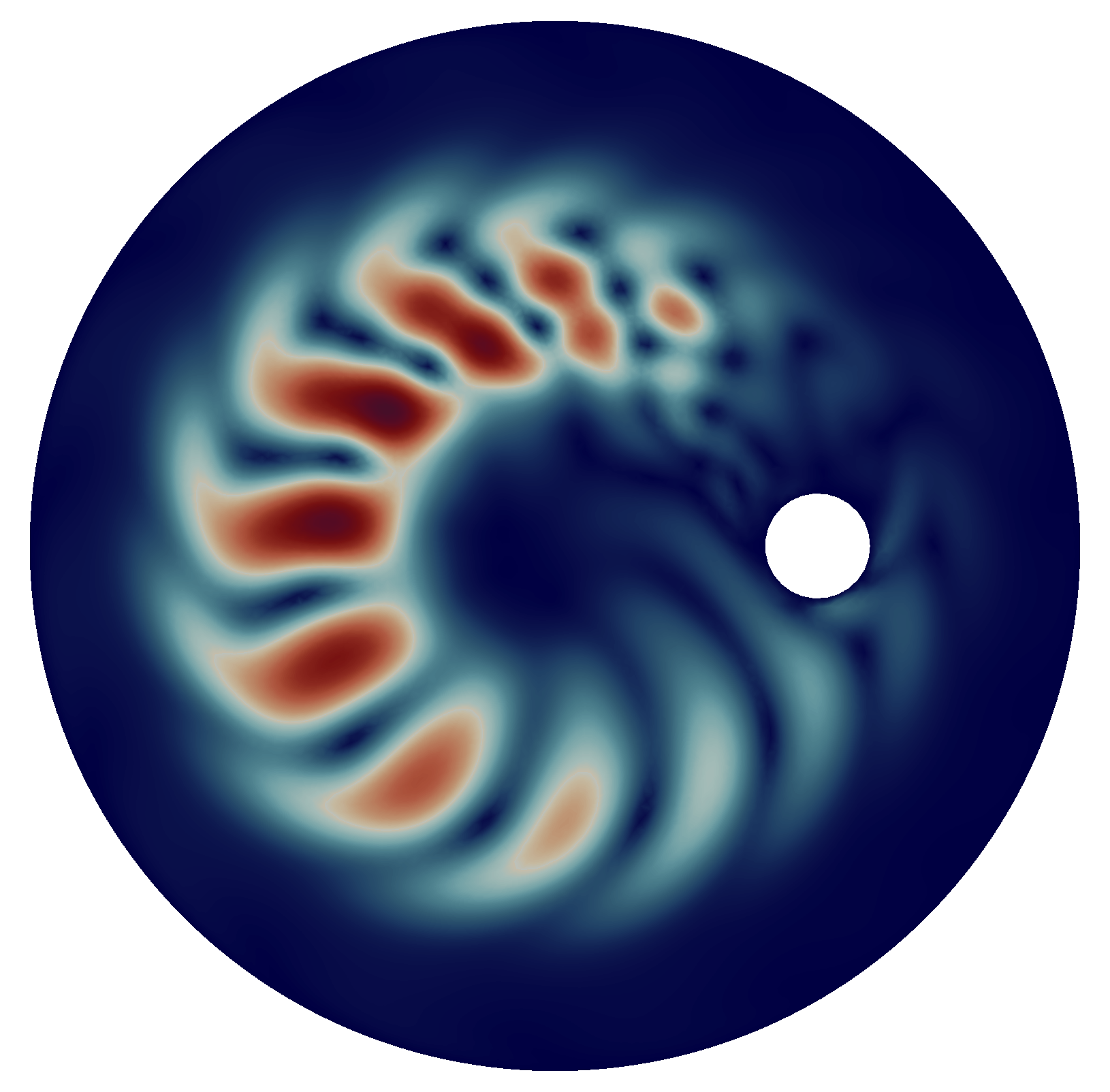}
		\end{subfigure}
		\begin{subfigure}{0.28\linewidth}
			\centering
			\includegraphics[width = 1\linewidth]{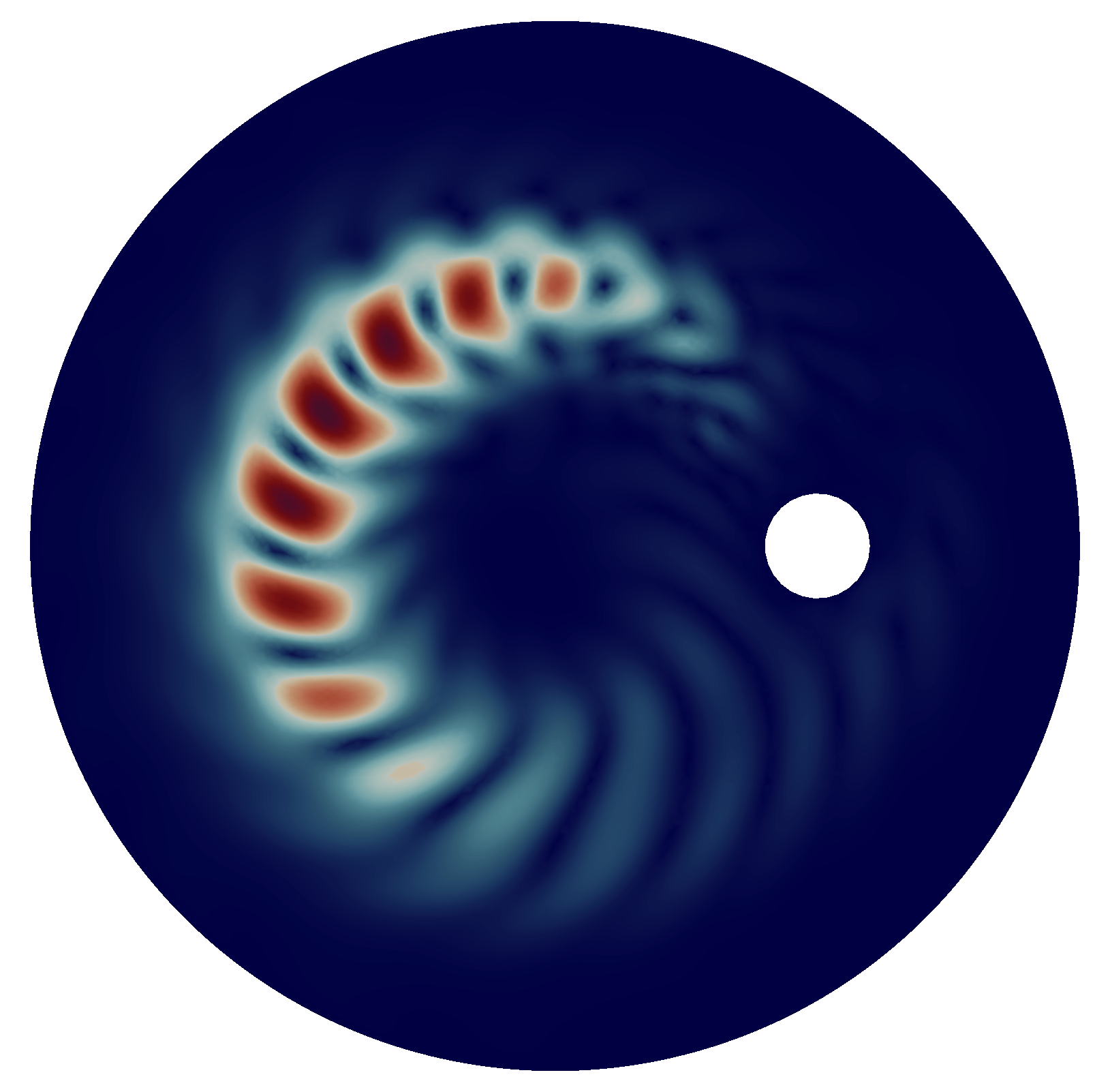}
		\end{subfigure}
		
		\caption{$\|u_R(x)\|$ with $R$ from 1 (top left) to 6 (bottom right).}
		\label{fig:v-basis}
	\end{figure}
	
	\begin{figure}[!ht]
		\centering
		\begin{subfigure}{0.28\linewidth}
			\centering
			\includegraphics[width = 1\linewidth]{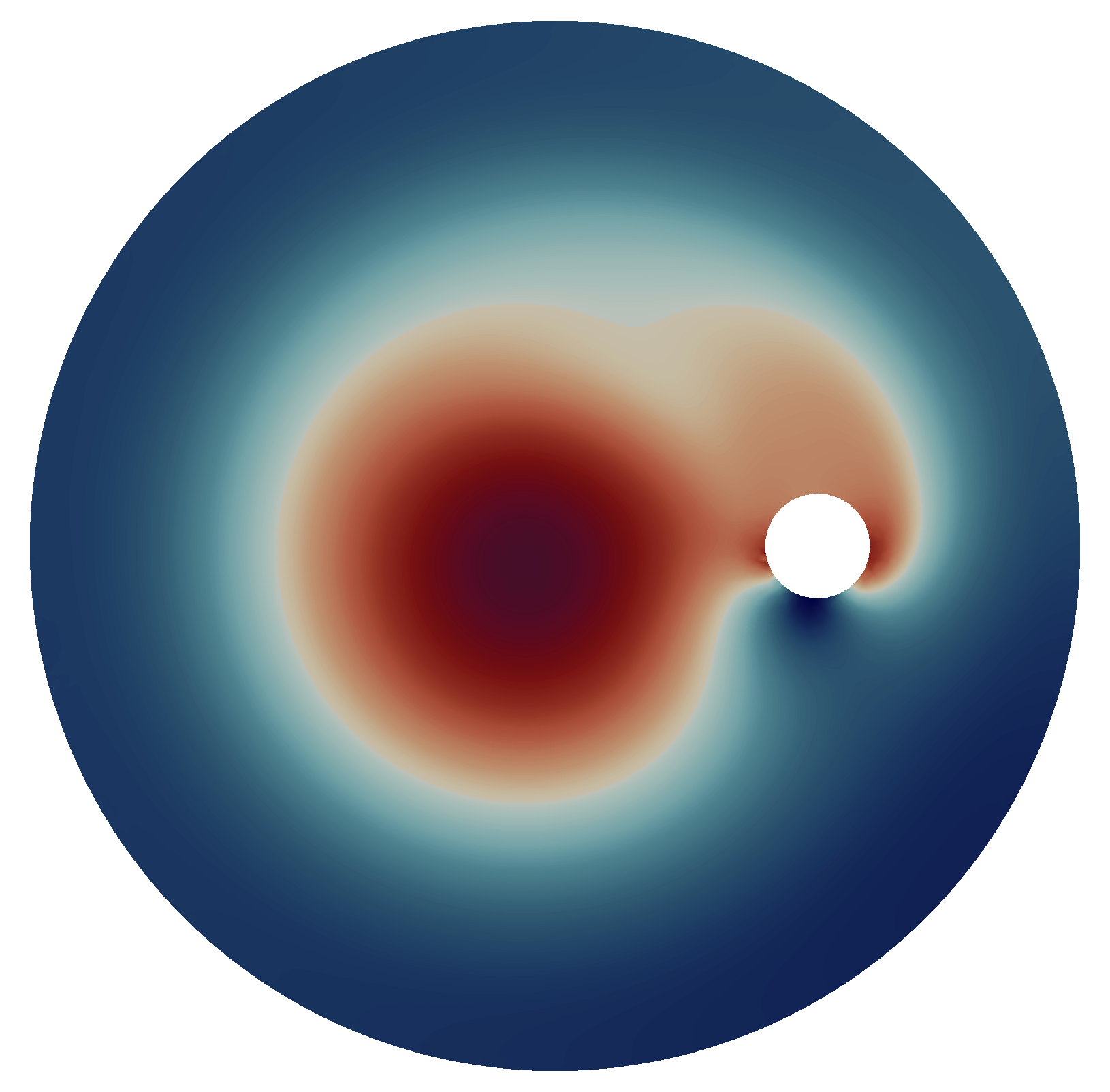}
		\end{subfigure}
		\begin{subfigure}{0.28\linewidth}
			\centering
			\includegraphics[width = 1\linewidth]{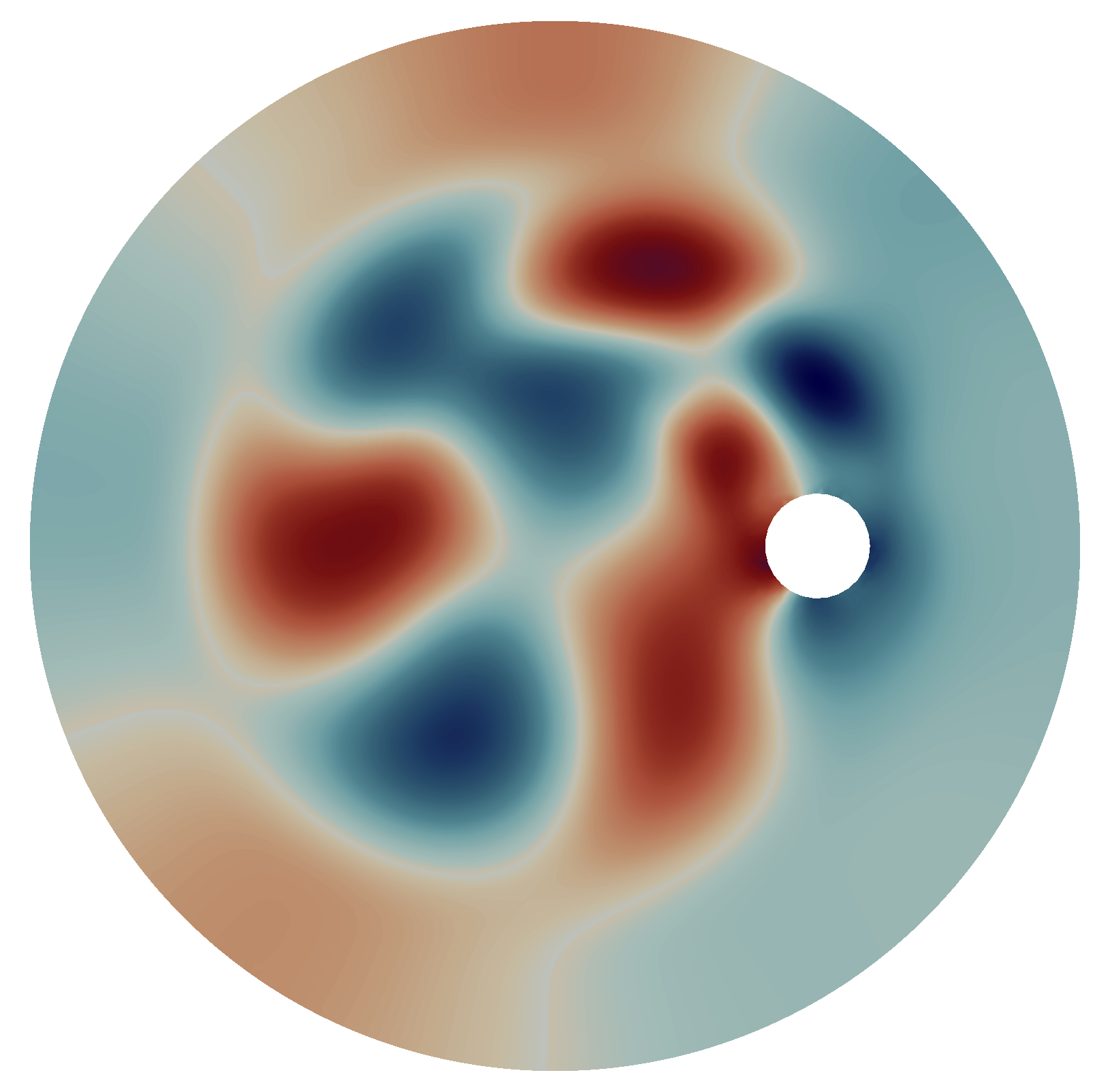}
		\end{subfigure}
		\begin{subfigure}{0.28\linewidth}
			\centering
			\includegraphics[width = 1\linewidth]{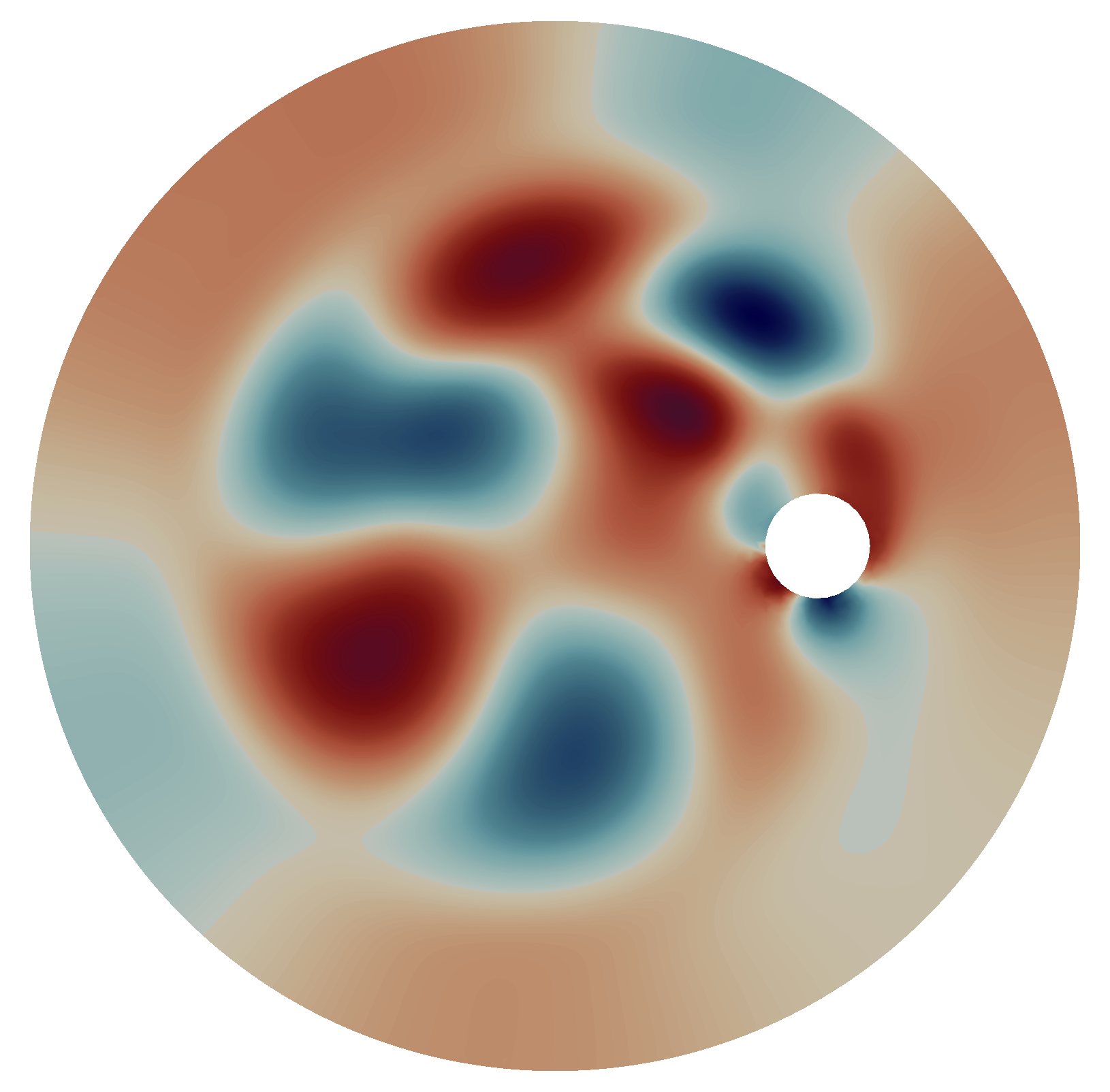}
		\end{subfigure}
		\begin{subfigure}{0.28\linewidth}
			\centering
			\includegraphics[width = 1\linewidth]{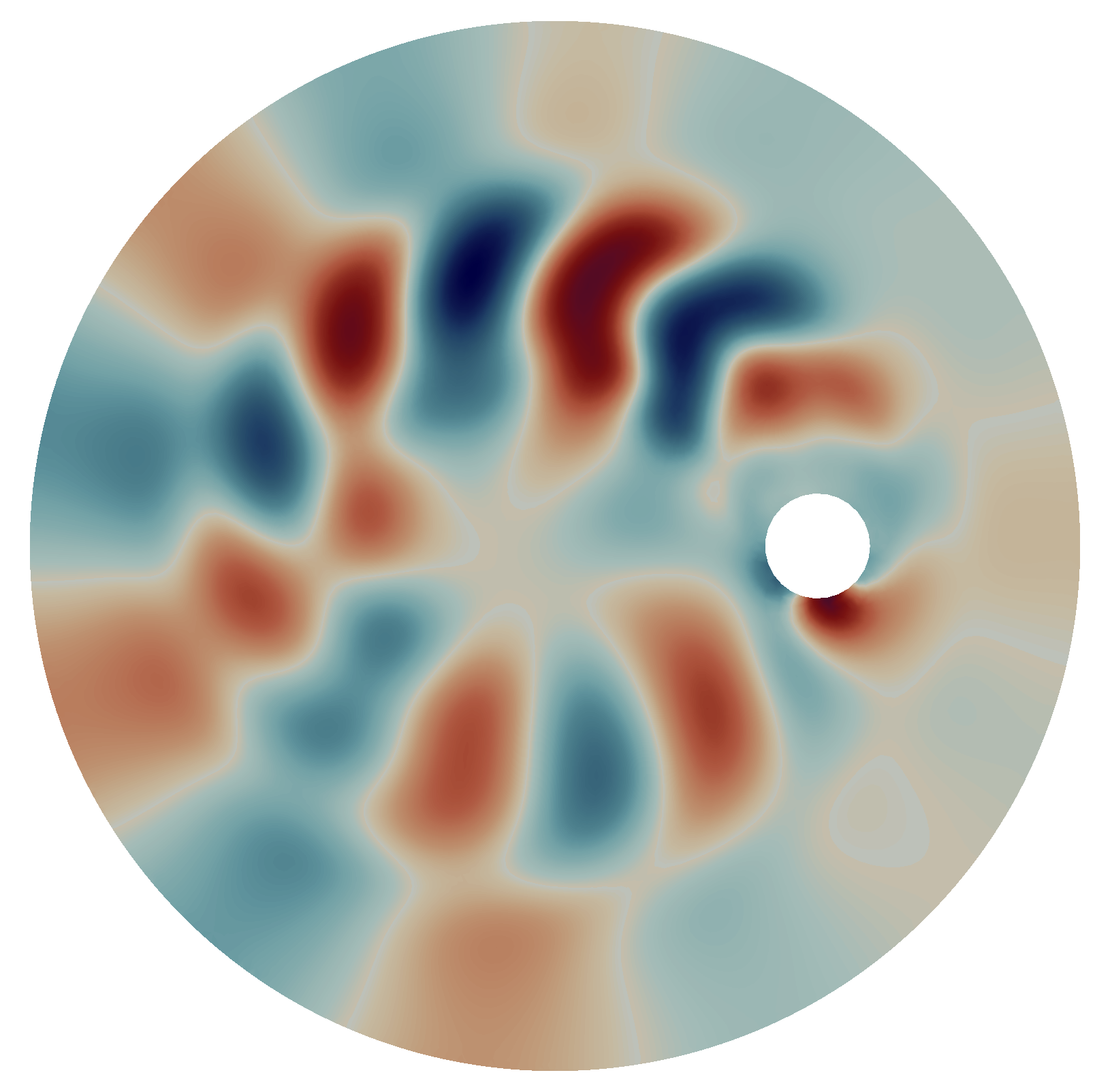}
		\end{subfigure}
		\begin{subfigure}{0.28\linewidth}
			\centering
			\includegraphics[width = 1\linewidth]{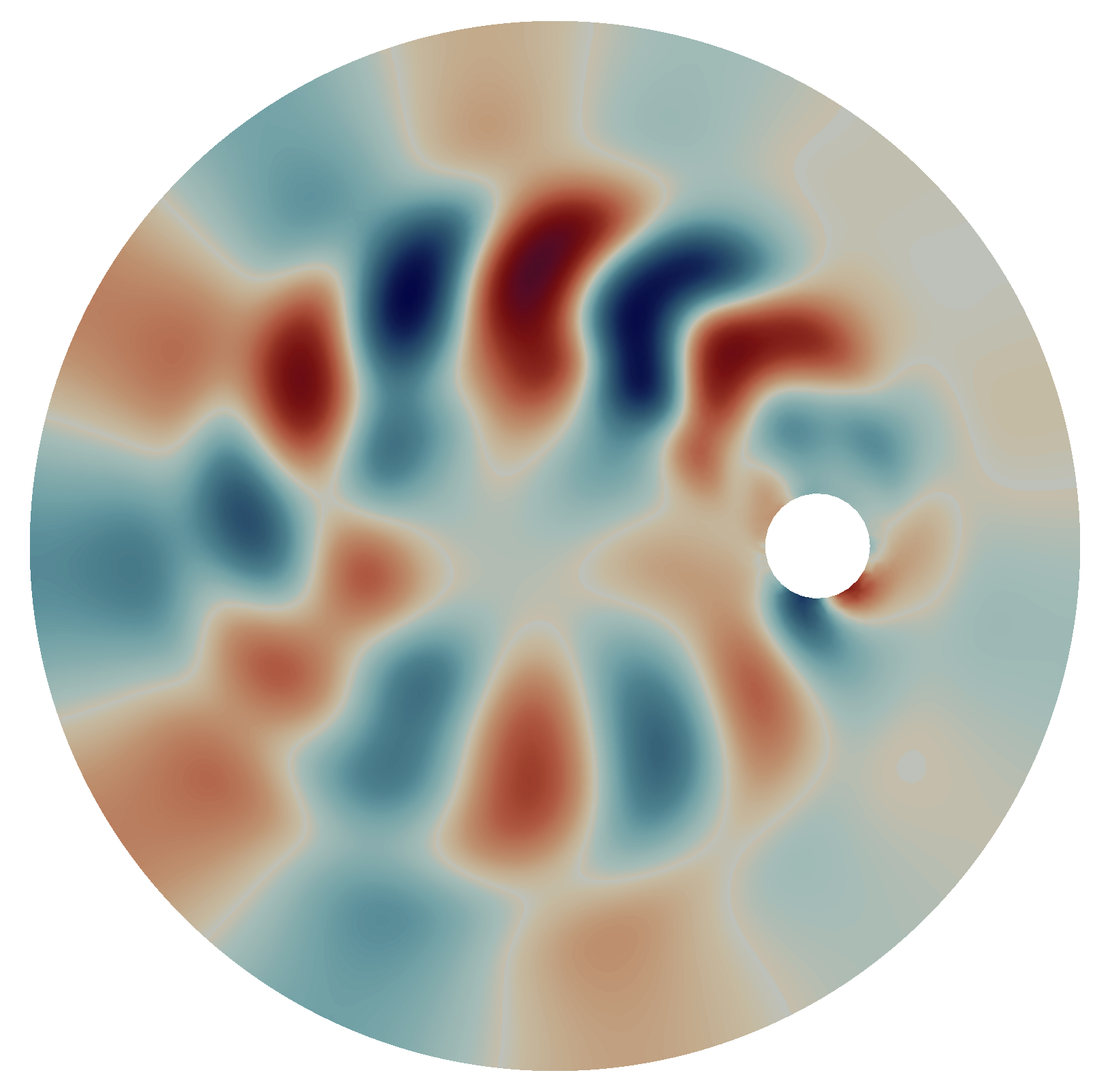}
		\end{subfigure}
		\begin{subfigure}{0.28\linewidth}
			\centering
			\includegraphics[width = 1\linewidth]{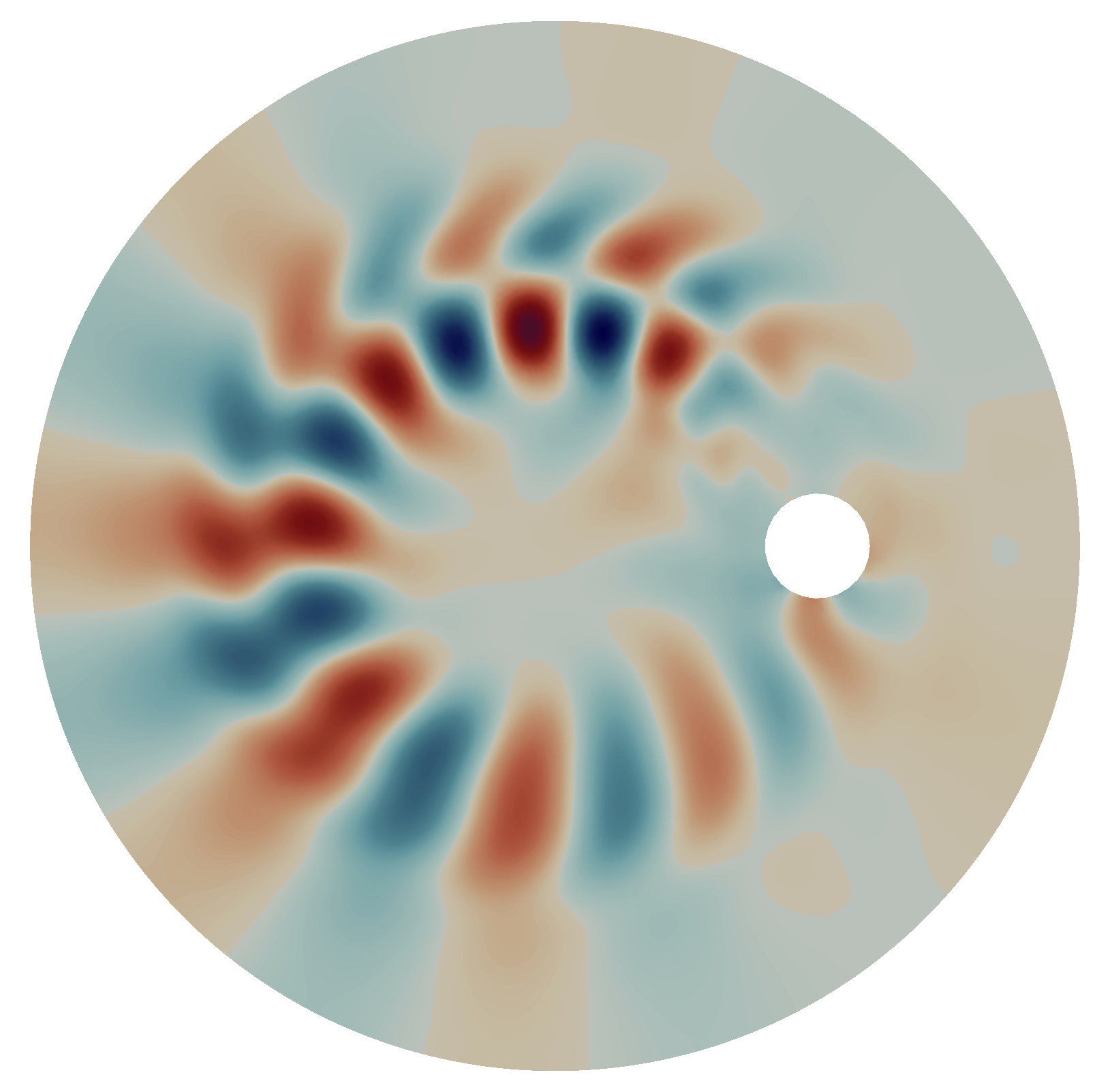}
		\end{subfigure}
		
		\caption{$p_R(x)$ with $R$ from 1 (top left) to 6 (bottom right).}
		\label{fig:p-basis}
	\end{figure}
	
	\begin{figure}[!ht]
		\centering
		\begin{subfigure}{0.28\linewidth}
			\centering
			\includegraphics[width = 1\linewidth]{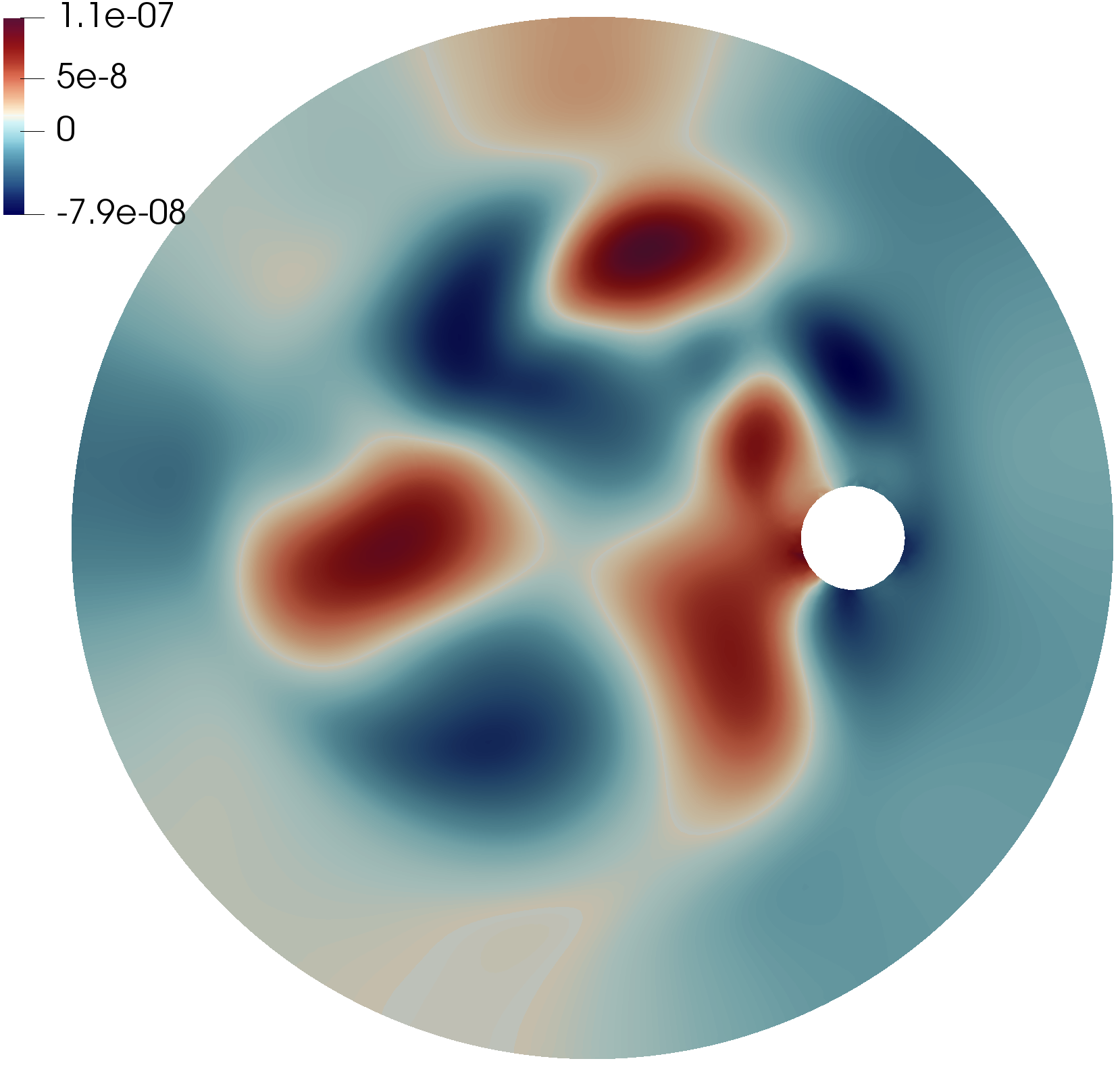}
		\end{subfigure}
		\begin{subfigure}{0.28\linewidth}
			\centering
			\includegraphics[width = 1\linewidth]{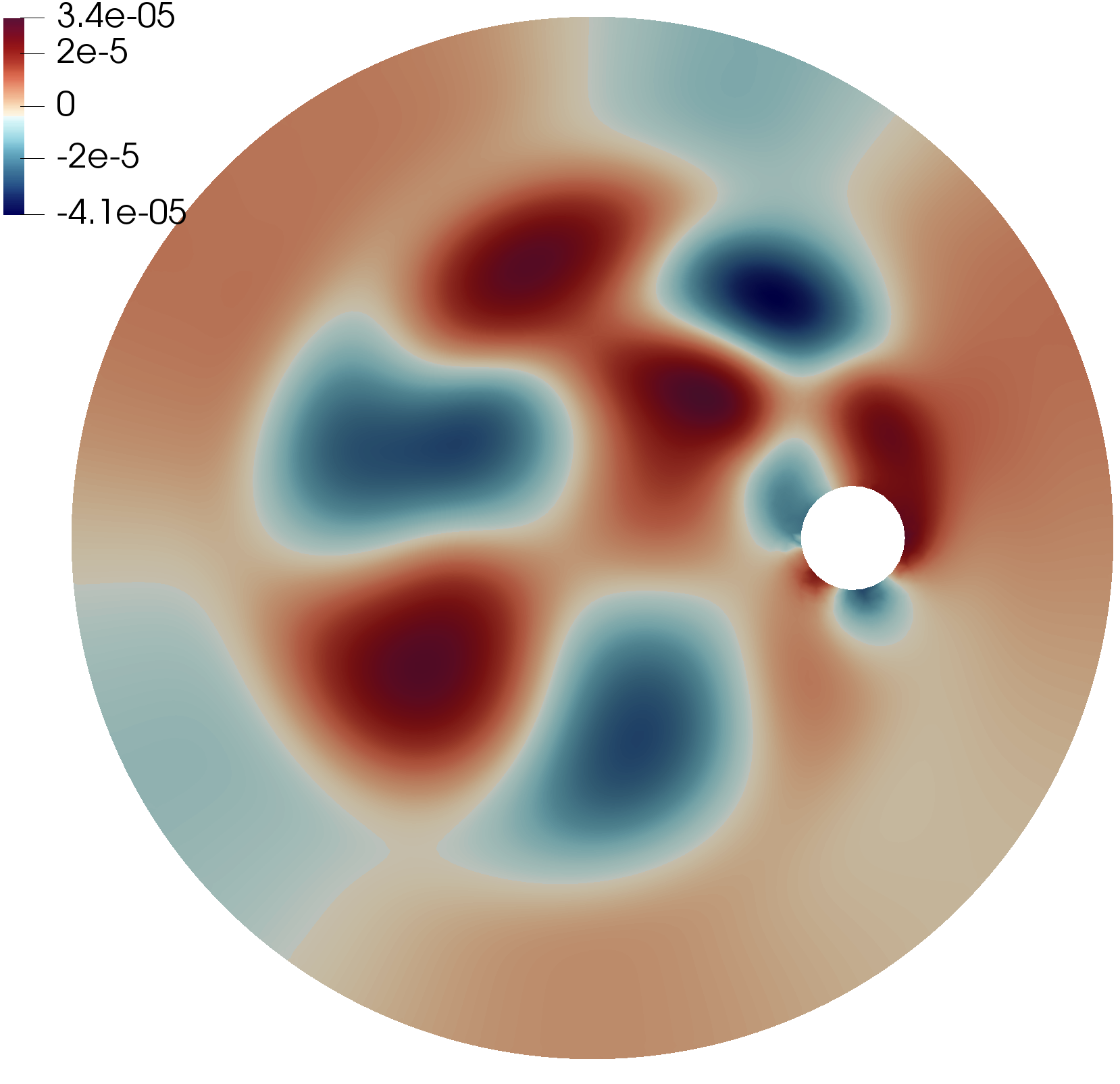}
		\end{subfigure}
		\begin{subfigure}{0.28\linewidth}
			\centering
			\includegraphics[width = 1\linewidth]{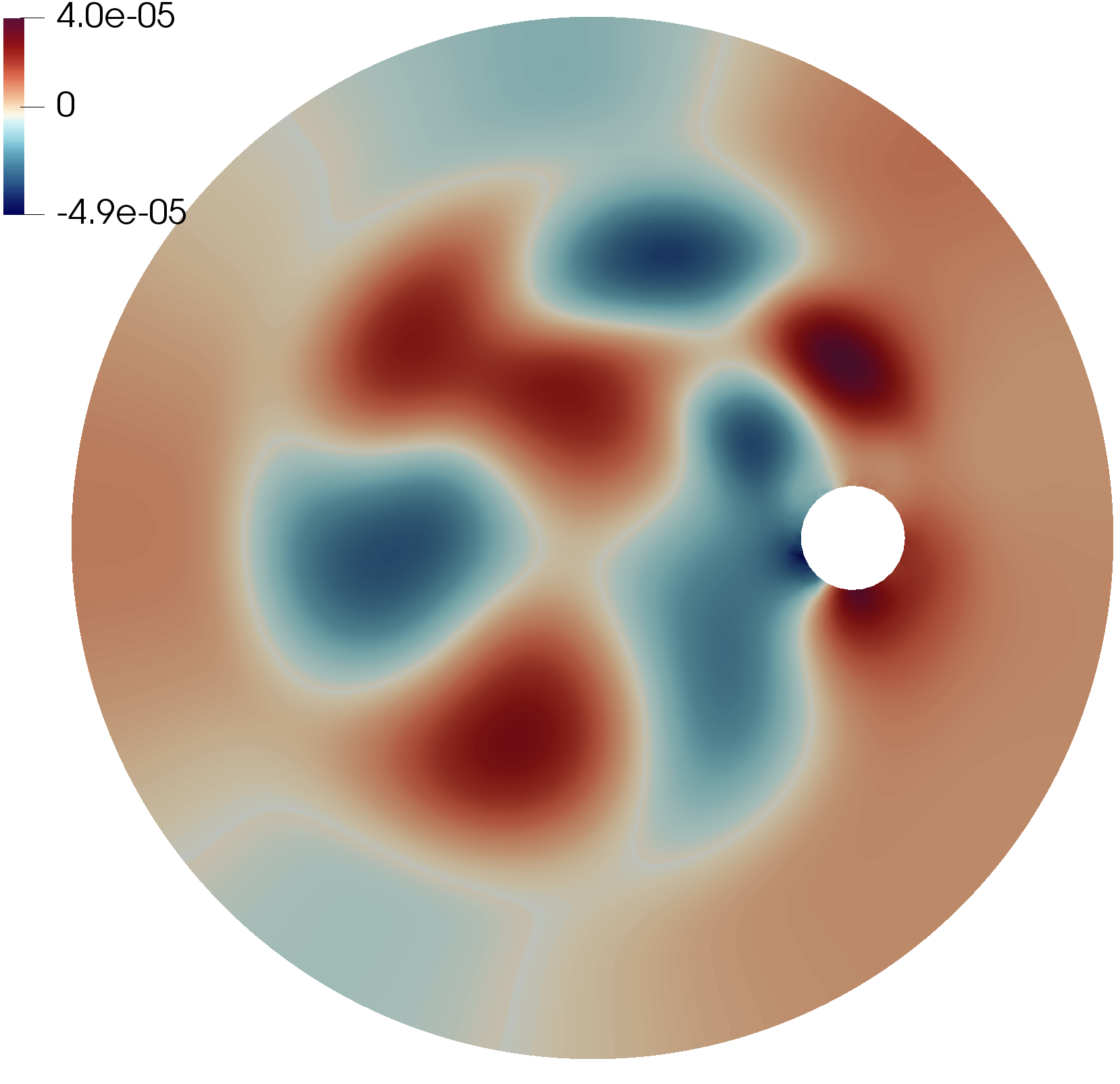}
		\end{subfigure}
		\begin{subfigure}{0.28\linewidth}
			\centering
			\includegraphics[width = 1\linewidth]{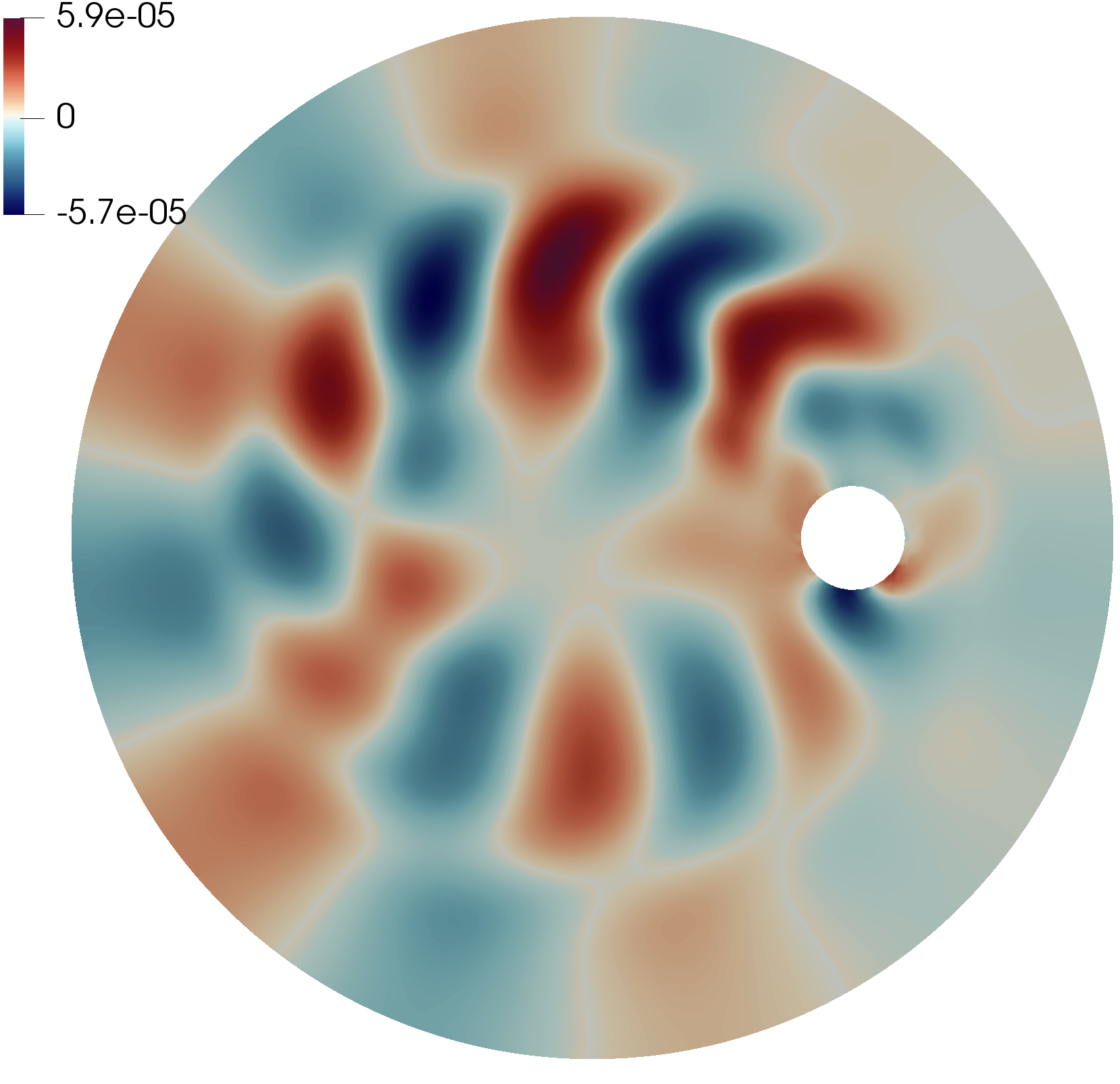}
		\end{subfigure}
		\begin{subfigure}{0.28\linewidth}
			\centering
			\includegraphics[width = 1\linewidth]{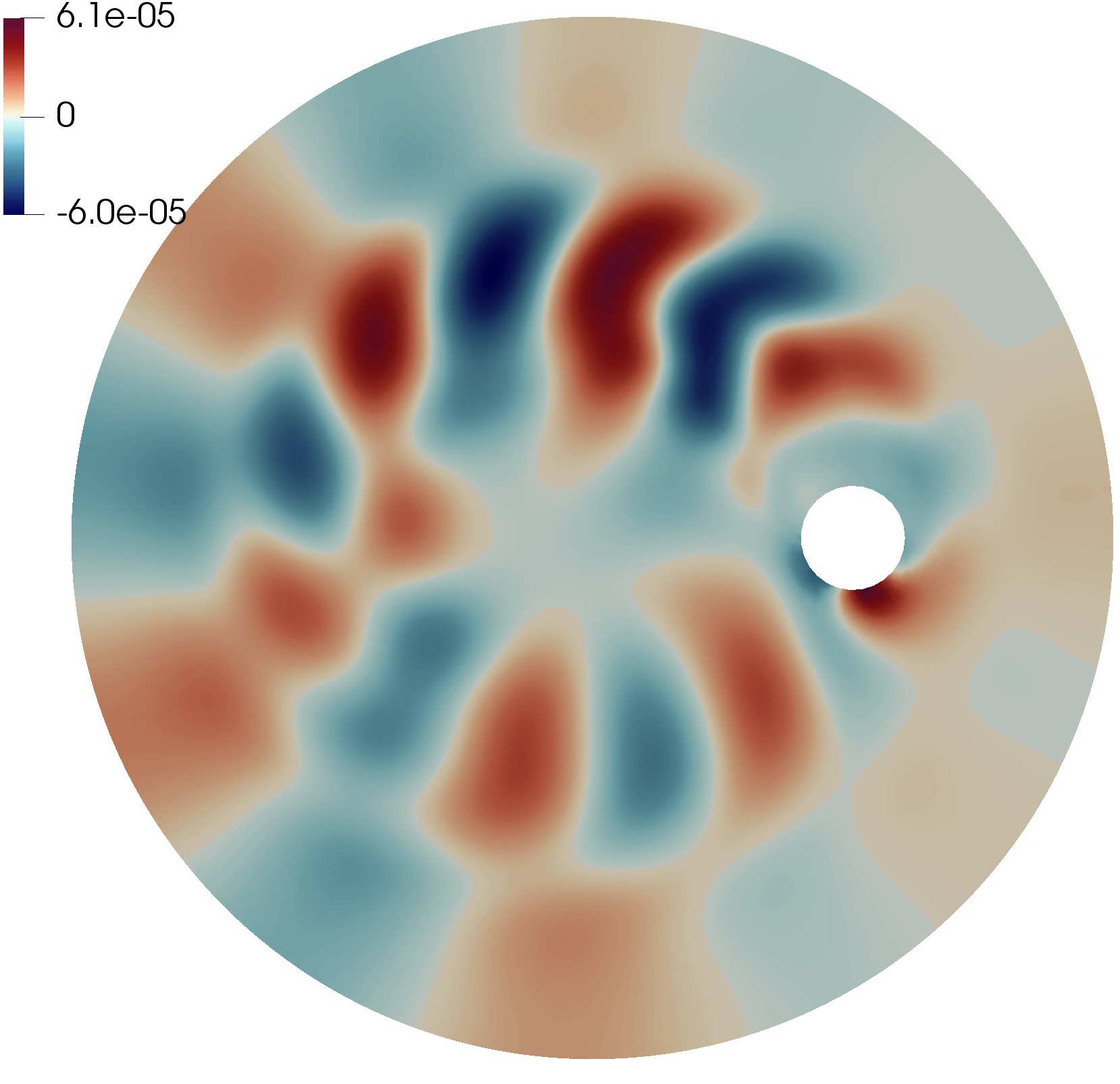}
		\end{subfigure}
		\begin{subfigure}{0.28\linewidth}
			\centering
			\includegraphics[width = 1\linewidth]{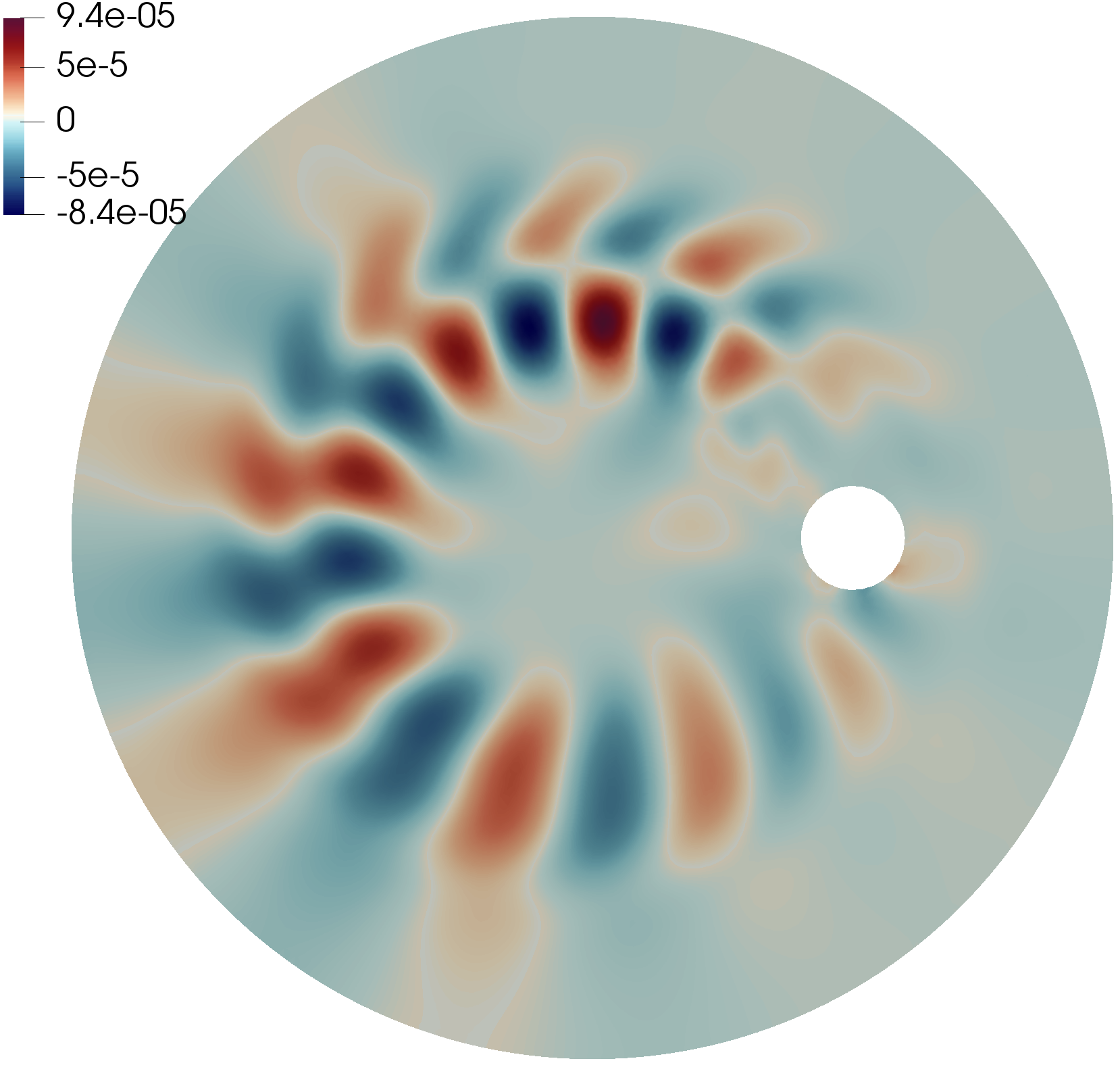}
		\end{subfigure}
		
		\caption{$\nabla \cdot u_R(x)$ with $R$ from 1 (top left) to 6 (bottom right).}
		\label{fig:div-basis}
	\end{figure}

	\begin{figure}[!ht]
		\centering
		\includegraphics[width = .4\linewidth]{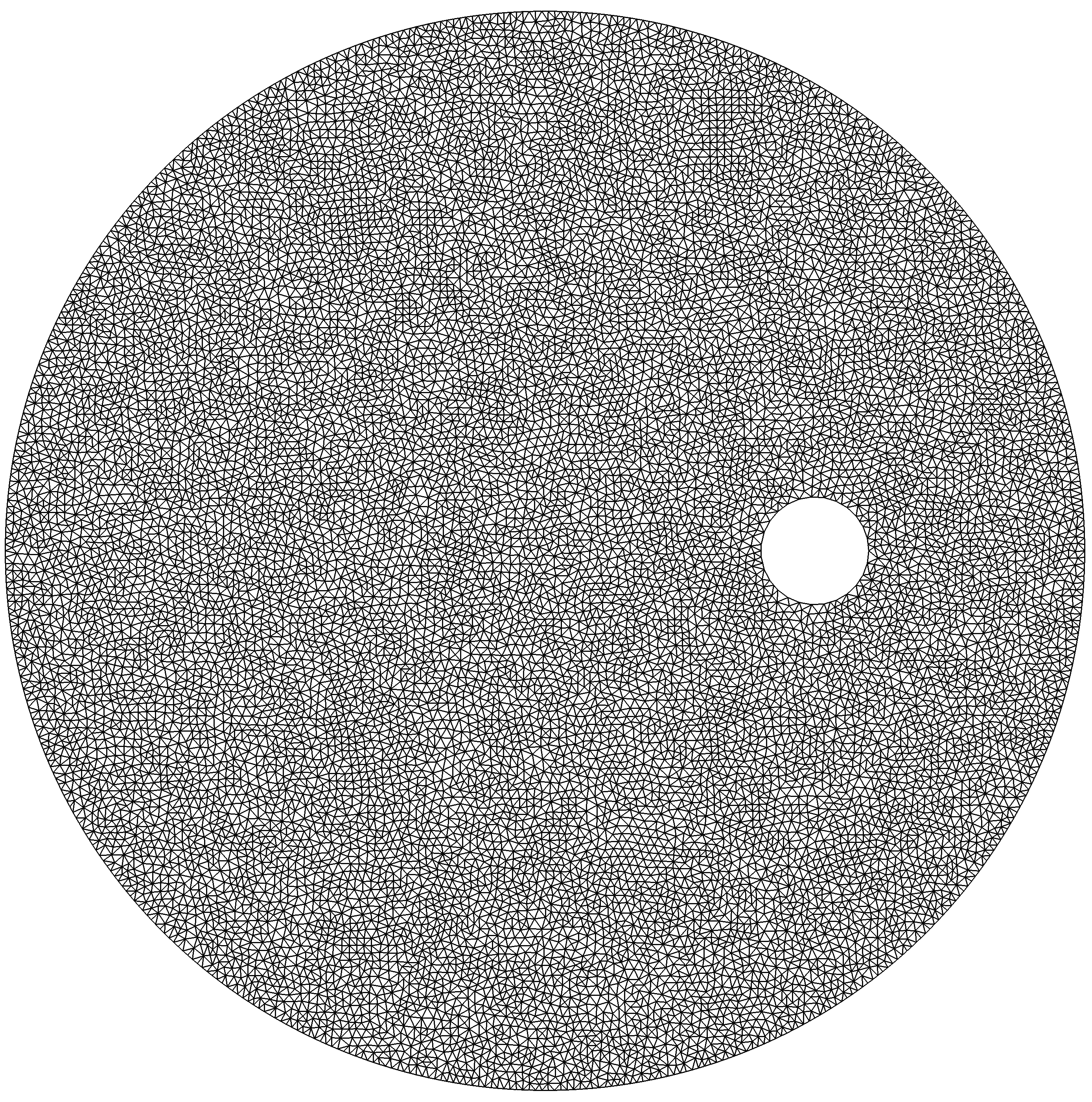}
		\caption{Spatial mesh for the finite element approximation.}
		\label{fig:mesh}
	\end{figure}

	\begin{figure}[!ht]
		\centering
		\begin{subfigure}{0.49\linewidth}
			\centering
			\includegraphics[width = 1\linewidth]{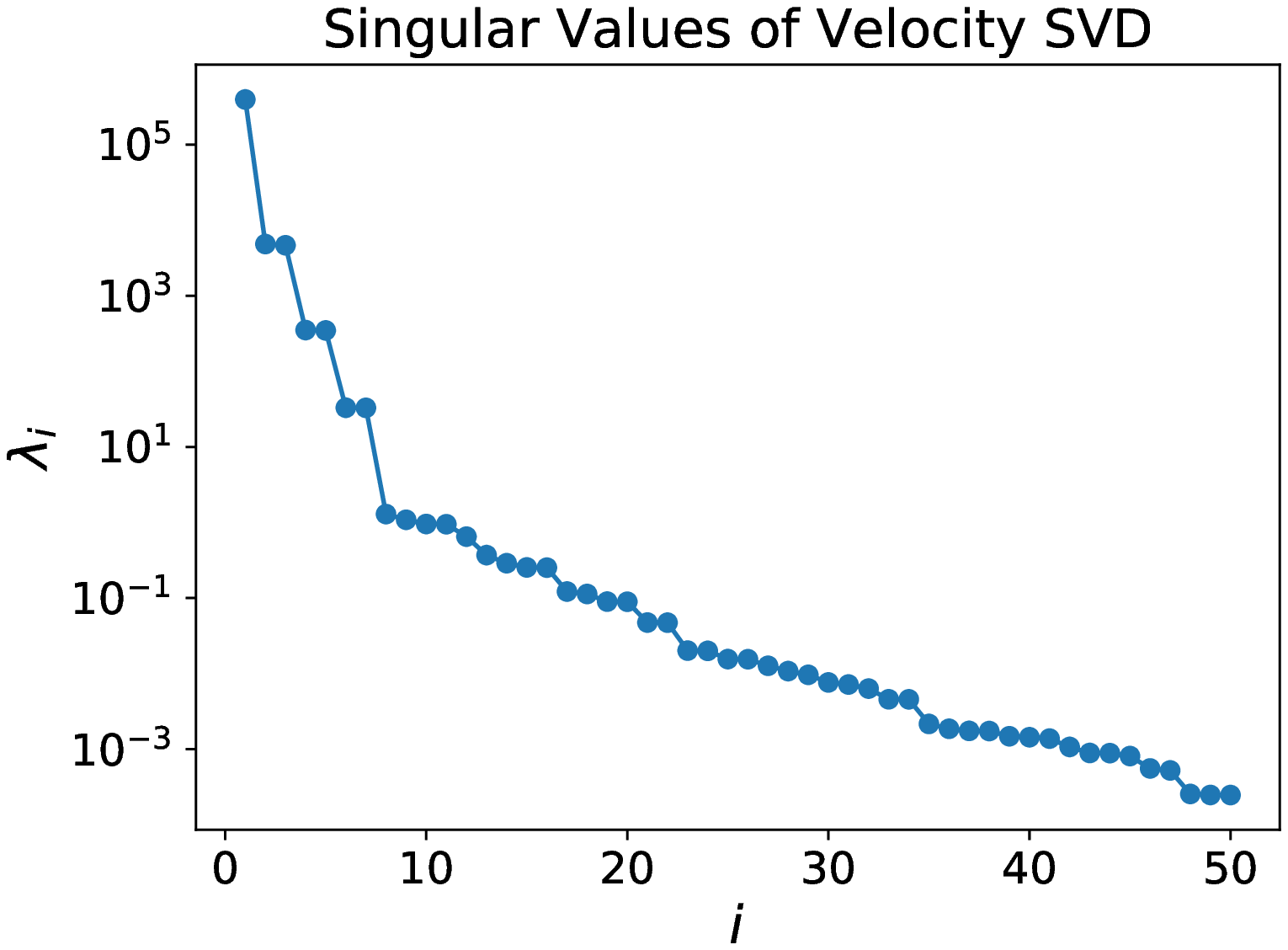}
		\end{subfigure}
		\begin{subfigure}{0.49\linewidth}
			\centering
			\includegraphics[width = 1\linewidth]{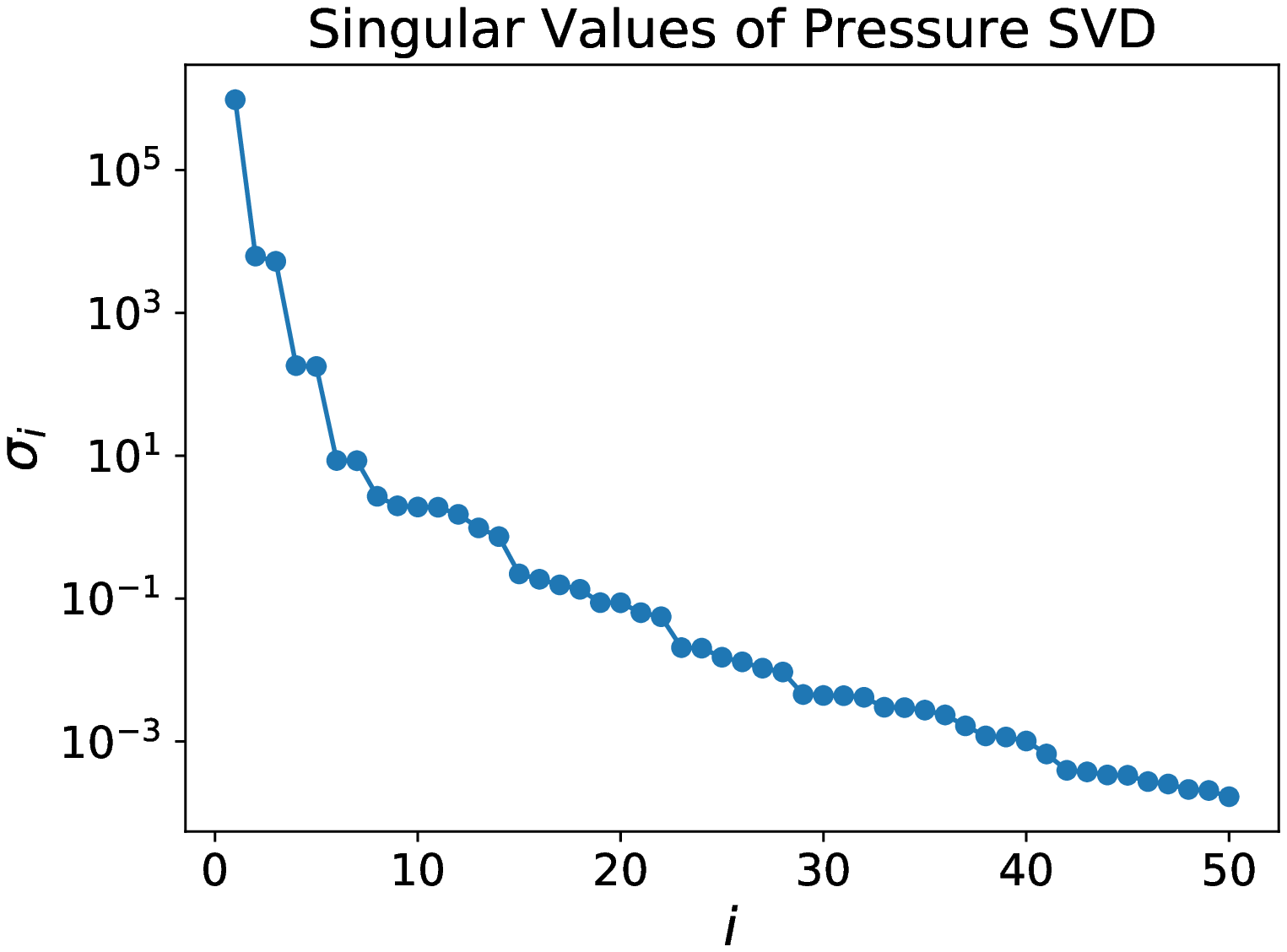}
		\end{subfigure}
		
		\caption{Singular values of the first 50 modes for pressure and velocity.}
		\label{fig:singular}
	\end{figure}
	
	The viscosity is $\nu = \frac{1}{100}$ and the body force is given by 
	\begin{equation}\notag
	f(x) = (-4y(1-x^2-y^2),4x(1-x^2-y^2)).
	\end{equation}
	The POD basis is computed from snapshots of the finite element discretization of a backward Euler artificial compression scheme. 
	\begin{remark}
		We emphasize that, since the snapshots are generated using an artificial compression scheme they will not be weakly divergence free. This is clearly illustrated in Fig. 6.3, where we plot the magnitude of the divergence of the velocity field.  We note that ROMs based on the pressure Poisson equation (i.e., ROMs in approach II in Section 1) cannot be used when the snapshots are not weakly divergence free.
	\end{remark}
	
	For the offline calculation, the flow is initialized at rest ($u_h^0\equiv 0$ and $p_h^0\equiv 0$). We discretized in space via the $P^2$-$P^1$ Taylor-Hood element pair. The spaces $X_h$ and $Q_h$ had 114,224 and 14,421 degrees of freedom respectively. We took $\Delta t=2.5e-4$ and $\varepsilon = 1e-6$. The mesh is shown in Fig. \ref{fig:mesh}. The no-slip, no-penetration boundary conditions are imposed on both cylinders. The flow developed into an almost periodic flow after $t=12$. Velocity and pressure snapshots were taken for every $t\in[12,16]$. The resulting singular values are shown in Fig. \ref{fig:singular}. The POD modes corresponding to the six largest singular values for velocity (resp. pressure) are shown in Fig. \ref{fig:v-basis} (resp. Fig. \ref{fig:p-basis}). 
	
	\begin{remark}\label{rem:lbb}
		We emphasize that the new AC-ROM uses the same number of velocity and pressure basis functions, i.e., $R=M$ in \eqref{AC_ROM}. Thus, we expect that the ROM LBB condition \eqref{eqn:inf-sup} is not satisfied.  This shows that the new AC-ROM avoids the ROM  LBB condition, which is generally prohibitively expensive for the RBM methods in approach I of Section 1 when those are used in realistic flows (see, e.g., Sections 4.2.2 and 4.2.3 in \cite{ballarin2015supremizer}).
	\end{remark}
	
	The force due to drag is the force exerted by the smaller cylinder against the main flow, which is counterclockwise. We calculated this as the line integral of the stress tensor around the smaller cylinder dotted with $(0,-1)$. The force due to lift is the line integral of the stress tensor around the smaller cylinder dotted with $(1,0)$. 
	
	With the stress tensor $\tau = \left(\nabla u + (\nabla u)^T\right) - pI$, and $\Gamma_{small}$ the boundary restricted to the inner cylinder, these quantities are
	\begin{equation}
	\text{Force due to drag} = -\int_{\Gamma_{small}} \tau ds \cdot e_2
	\end{equation}
	\begin{equation}
	\text{Force due to lift} = \int_{\Gamma_{small}} \tau ds \cdot e_1.
	\end{equation}
	
	\subsection{Numerical Investigation}
	
	We compare the kinetic energy, force due to drag, and force due to lift of the ROM simulations with $R=M=3,5$ and $7$ with the offline simulation in Fig. \ref{fig:energy}. $R\geq 7$ appears sufficient to capture the kinetic energy, lift and drag accurately. Again, this is in spite of the fact that the LBB condition is not satisfied due to using an equal number of pressure and velocity modes (see Remark \ref{rem:lbb}).

	\begin{figure}[!ht]
		\centering
		\begin{subfigure}{0.49\linewidth}
			\centering
			\includegraphics[width = 1\linewidth]{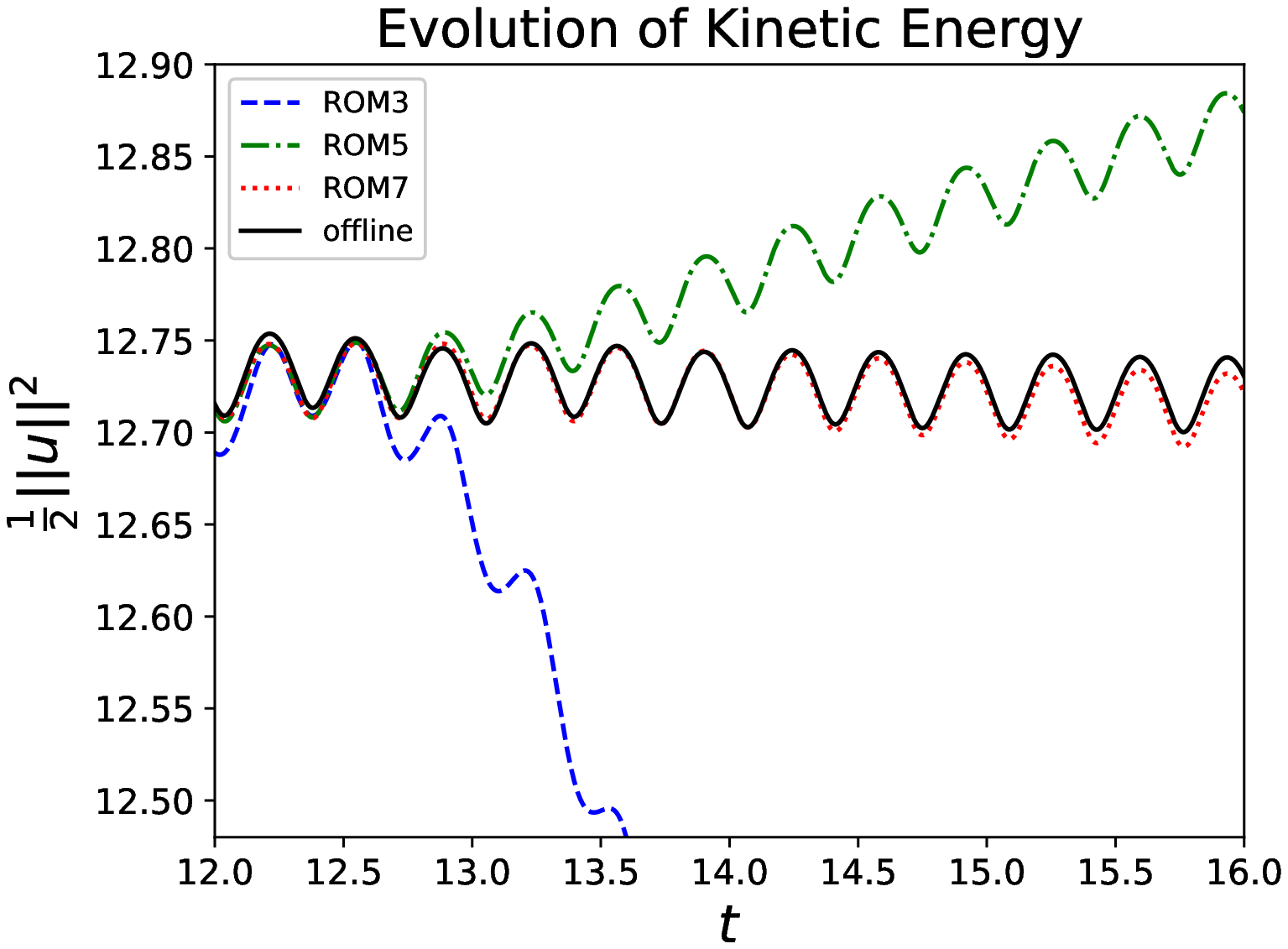}
		\end{subfigure}
		\begin{subfigure}{0.49\linewidth}
			\centering
			\includegraphics[width = 1\linewidth]{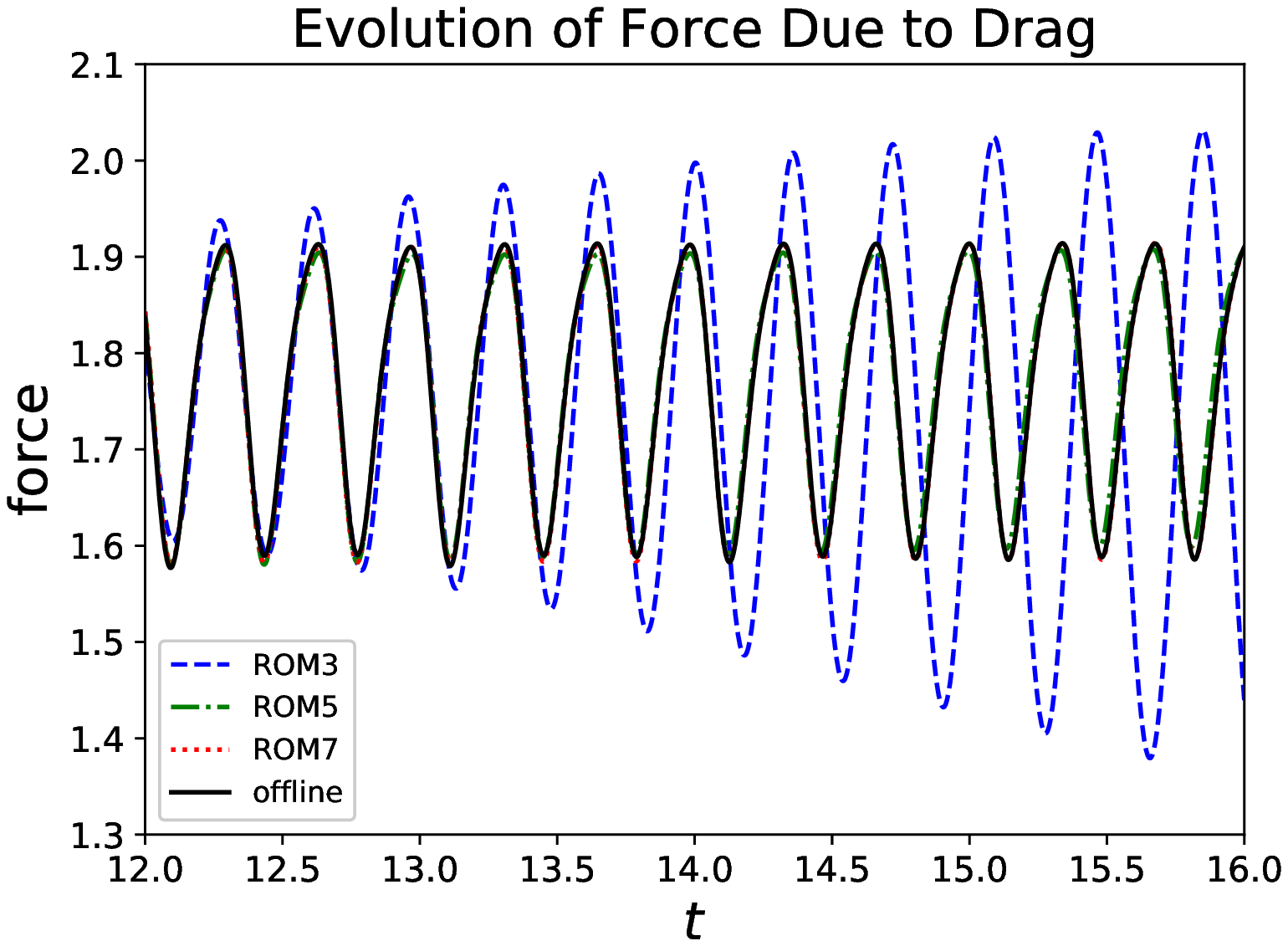}
		\end{subfigure}
		\begin{subfigure}{0.49\linewidth}
			\centering
			\includegraphics[width = 1\linewidth]{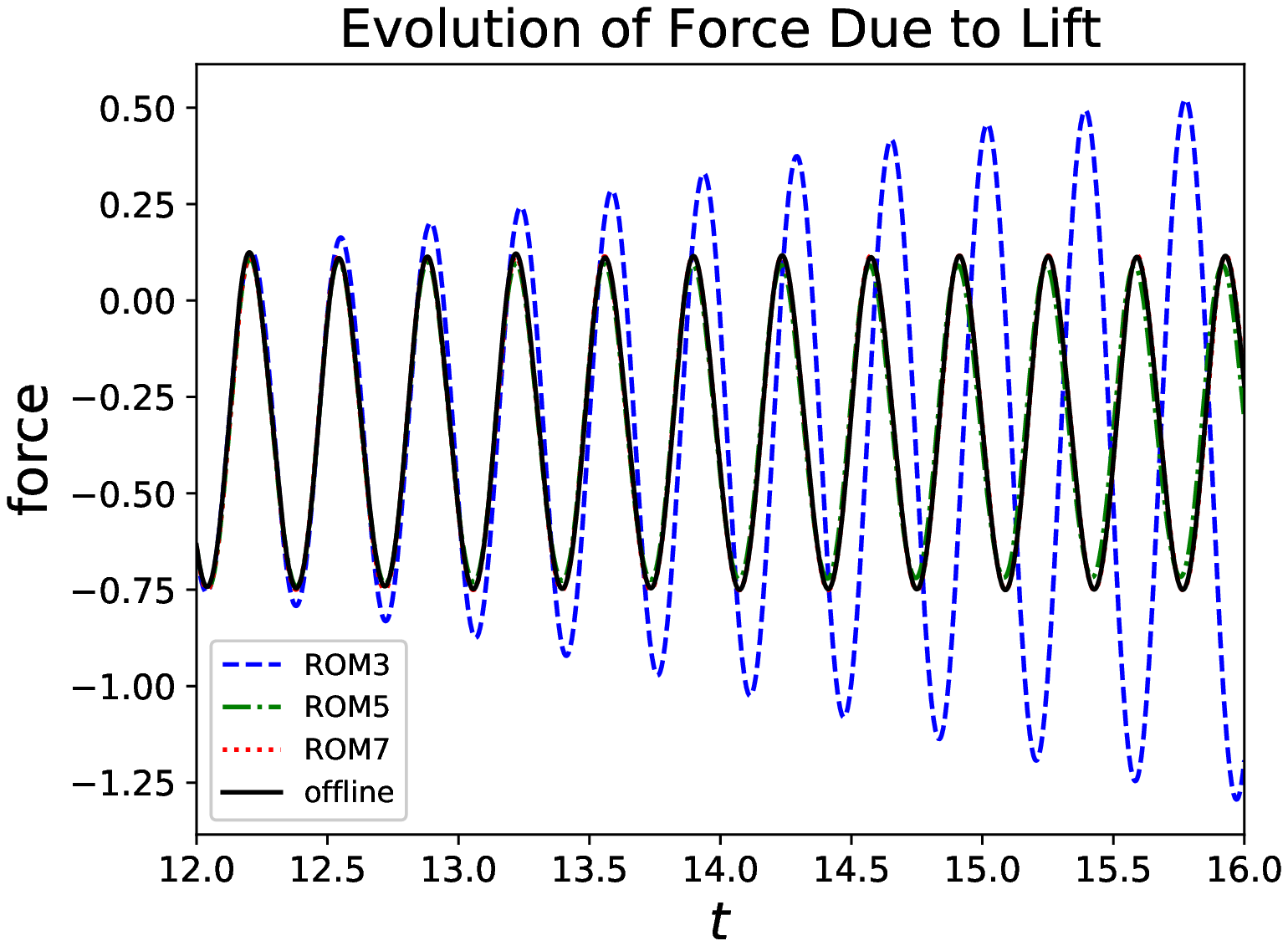}
		\end{subfigure}
		\caption{Evolution of the energy, drag and lift for AC-ROM with varying basis cardinality compared to the benchmark.}
		\label{fig:energy}
	\end{figure}
	
	Next, to illustrate numerically the theoretical scalings proved in Theorem \ref{theorem:1}, we show that the new AC-ROM algorithm yields first order scalings with respect to the time step, $\Delta t$. To test convergence with respect to $\Delta t$, we take $R=50$ using the snapshots for $t\in [12,16]$ with the offline simulation done from $t=12$ to $t=12.24$. $\Delta t$ ranges from $1.6e-3$ to $2.5e-4$, which was the stepsize from the offline simulation. The error is measured by comparing the $u_R$ to the corresponding offline solution $u_h$. The relative $l^2L^2$ errors that are shown in Fig. \ref{fig:error} verify the $\mathcal{O}(\Delta t)$ convergence proven in Theorem \ref{theorem:1}.
	
	We briefly outline the process of computing the first principal angle between the spaces $X_{R}^{div}$ and $Q_{M}$.
	Let $\{\nabla \cdot \varphi^{orth}_{i}\}_{i=1}^{R}$ denote the orthonormalized  basis of $X_{R}^{div}$ \eqref{X-div}. We consider the matrices $\mathbb{Q}= [\psi_{1},\psi_{2},\ldots \psi_{M}]$ and $\mathbb{X} = [\nabla \cdot \varphi^{orth}_{2},\nabla \cdot \varphi^{orth}_{2},\ldots \nabla \cdot \varphi^{orth}_{R}]$. Multiplying these two matrices and taking the SVD gives 
	\begin{equation}
	\mathbb{X}^{\top}\mathbb{Q} = U \Sigma V
	\end{equation}  
	The first principal angle will then be given in terms of the first nonzero entry of $\Sigma$, by $\theta_{1} = \arccos(\sigma_{1})$. We measured the influence of the principal angle between the velocity and pressure POD basis using the method outlined above. The results are shown in Fig. \ref{fig:beta}. $\alpha^2$ begins near machine precision and seems to plateau around $1e-2$ when adding more basis functions. This appears to match up with our theoretical results and explains why we do not observe an order reduction in our numerical investigation.
	\begin{figure}[!ht]
		
		\begin{subfigure}{0.49\linewidth}
			\centering
			\includegraphics[width = .95\linewidth]{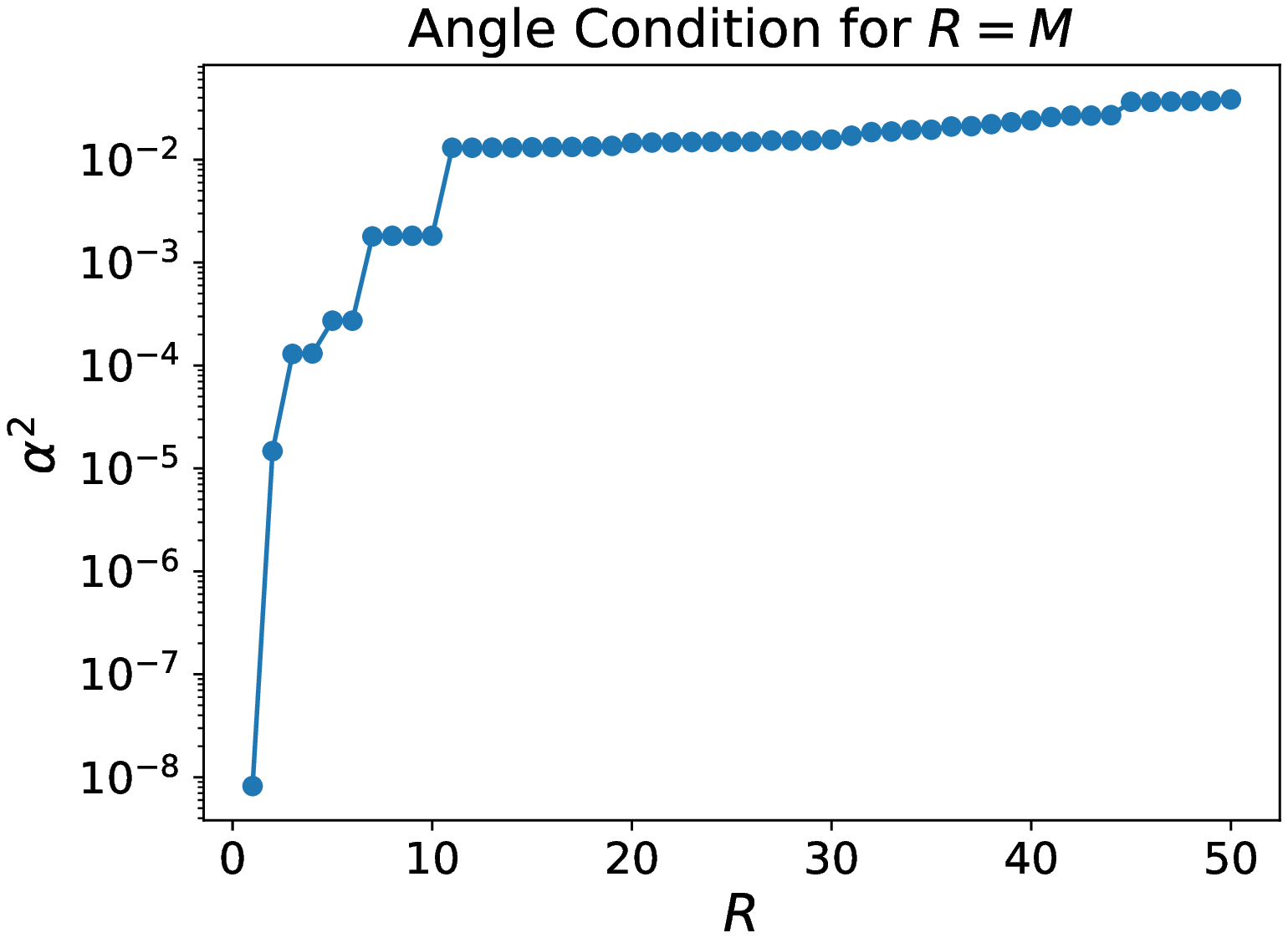}
		\end{subfigure}
		\begin{subfigure}{0.49\linewidth}
			\centering
			\includegraphics[width = .95\linewidth]{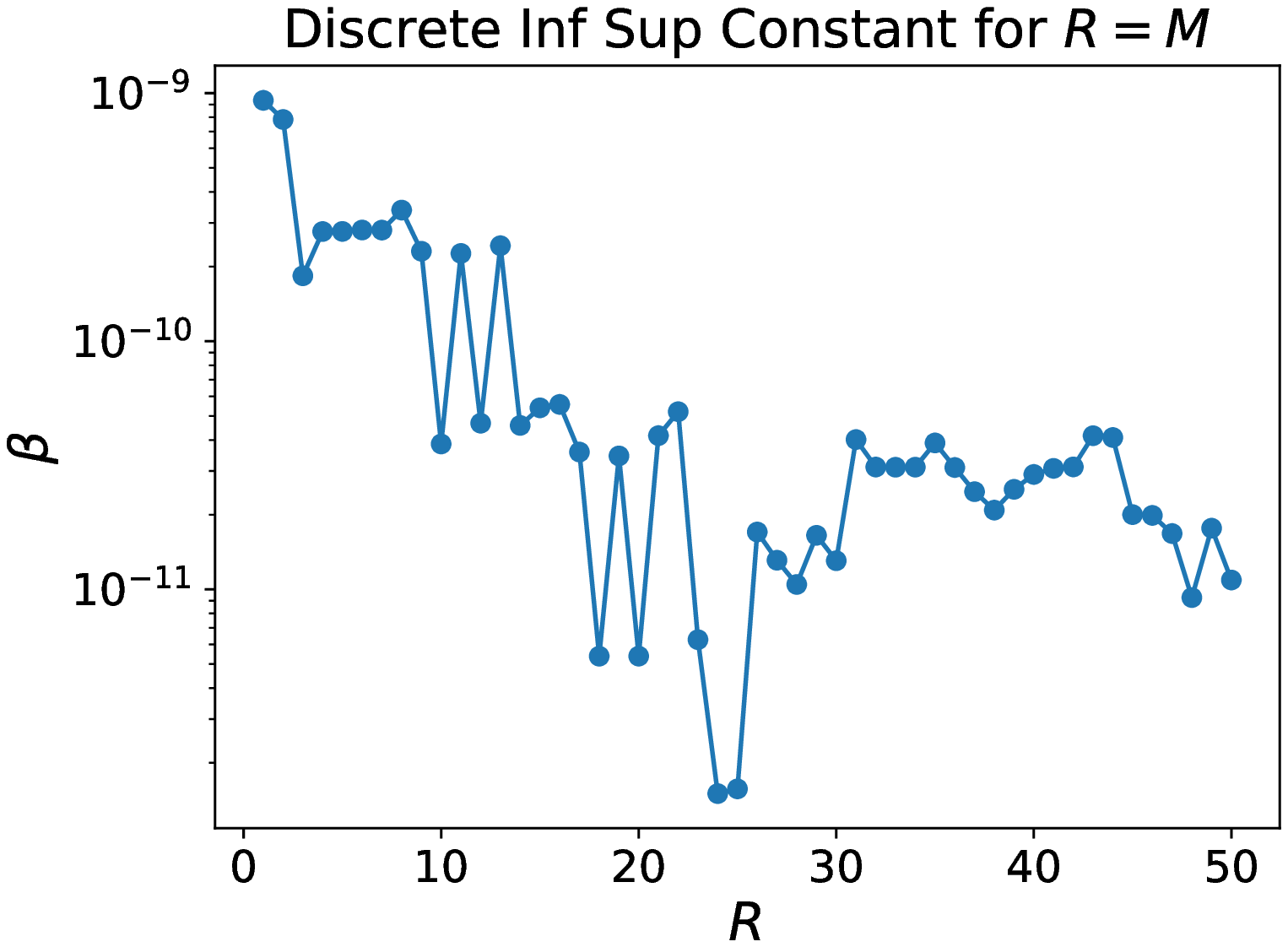}
		\end{subfigure}
		\caption{Value of $\alpha^{2}$ for equal number of velocity and pressure basis functions on the left and the corresponding inf-sup constant on the right.
			\label{fig:beta}}
	\end{figure}

	\begin{figure}[!ht]
		\centering
		\begin{subfigure}{0.49\linewidth}
			\centering
			\includegraphics[width = .95\linewidth]{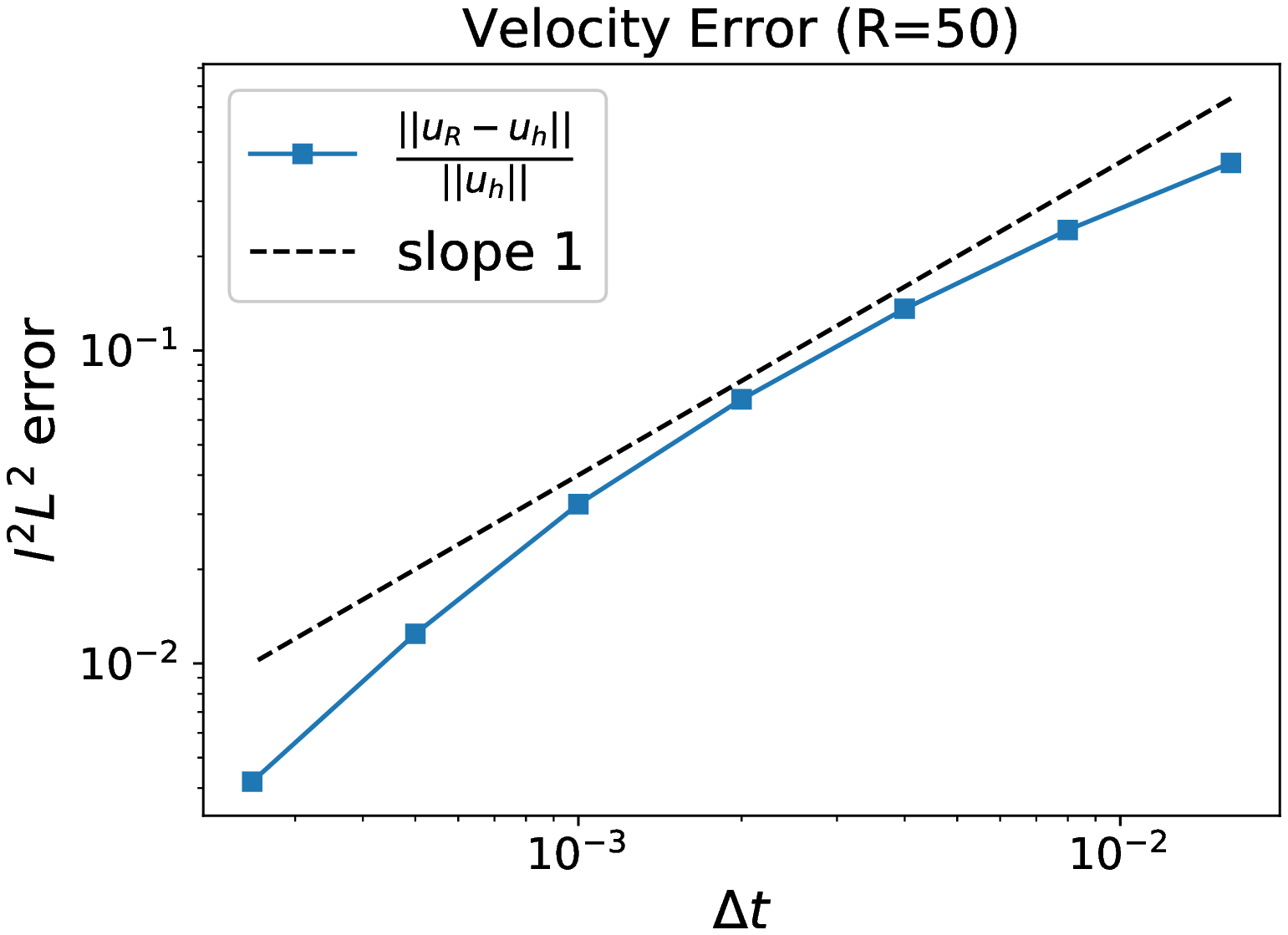}
		\end{subfigure}
		\begin{subfigure}{0.49\linewidth}
			\centering
			\includegraphics[width = .95\linewidth]{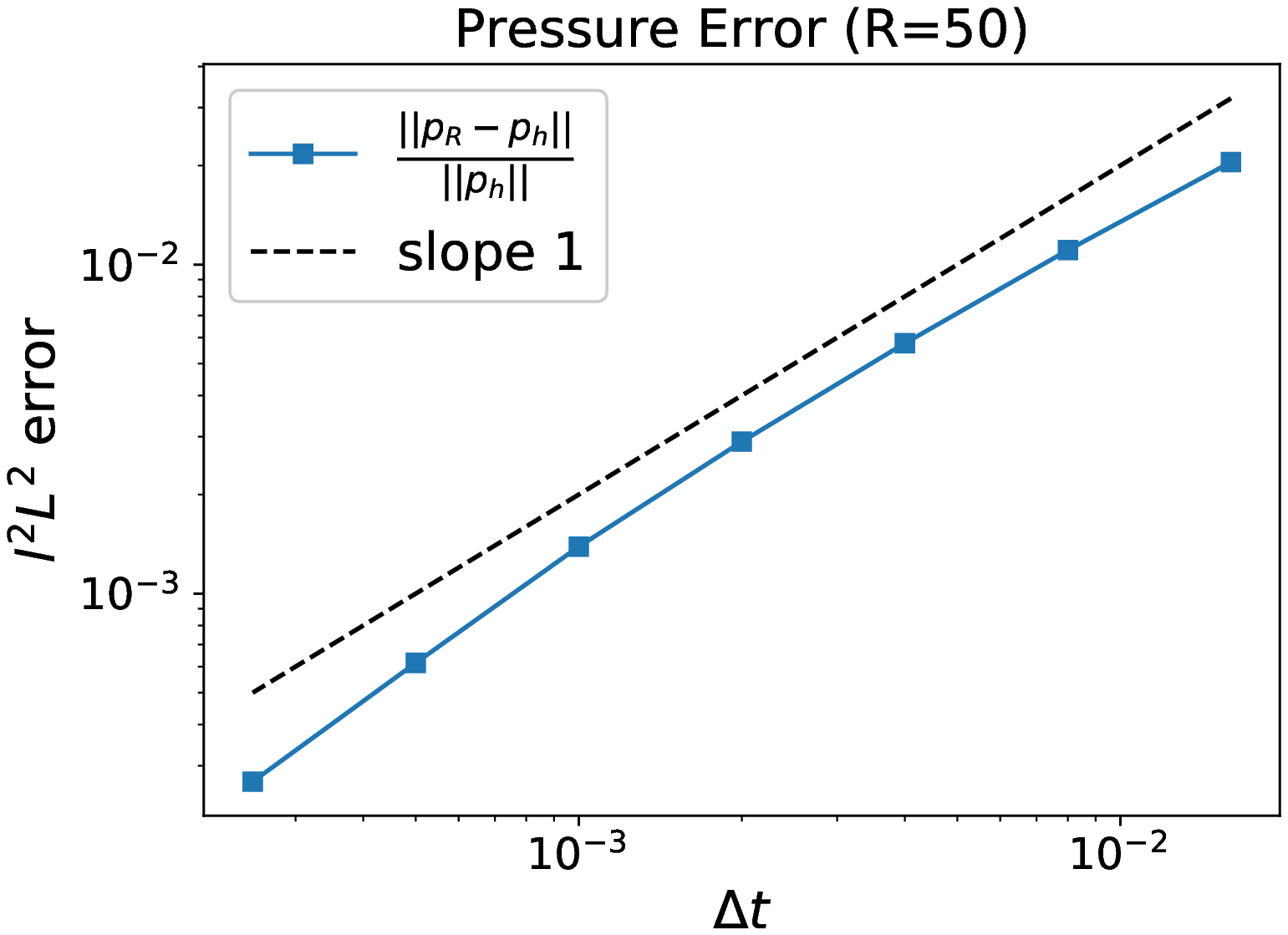}
		\end{subfigure}
		\caption{Both the pressure and velocity are first order convergent.}
		\label{fig:error}
	\end{figure}

	
	\section{Conclusions and Outlook}\label{sec:conclusions}
	In this paper, we propose an artificial compression reduced order model (AC-ROM) for the numerical simulation of fluid flows.  The new AC-ROM provides approximations for both the velocity and the pressure.  Compared to the current ROMs that generate pressure approximations, the new AC-ROM has two main advantages: (i) it does not require the fulfillment of the inf-sup/LBB condition, which can be prohibitively expensive in current ROMs~\cite{caiazzo2014numerical}; and (ii) it does not require weakly divergence-free snapshots, which allows it to work with snapshots generated with, e.g., artificial compression, penalty, or projection methods.  
	
	In Section 4, we prove the unconditional stability of the finite element discretization of the  new AC-ROM.  In Section 5, we prove an error estimate for the AC-ROM. In particular, we show that that it is possible to overcome the $\Delta t^{-1}$ order degradation due to lack of inf-sup stability,if the angle between the divergence of the velocity space and pressure space is sufficiently small.
	
In Section 6, we perform a numerical investigation of the new AC-ROM for a two dimensional flow between two offset cylinders. To generate the snapshots, we use the artificial compression method. Thus, the snapshots used in the AC-ROM construction are not weakly divergence-free, which is illustrated in \ref{fig:div-basis}. We also show that the velocity and pressure spaces of the new AC-ROM do not satisfy the LBB condition (see Fig. \ref{fig:beta}). In the numerical investigation of the new AC-ROM, we first show that the AC-ROM yields accurate velocity and pressure approximations.  Specifically, in Fig.  \ref{fig:energy}, we show that it provides energy, drag force, and lift force approximations that are close to the direct numerical simulation results.  Next, to illustrate numerically the theoretical scalings proved in Section 5, we show that the new AC-ROM algorithm yields first order scalings with respect to the time step.  Finally, in Fig. \ref{fig:beta}, we show that the constant multiplying the $\Delta t^{-1}$ term in the error estimate is extremely small. This may explain why we do not observe an order reduction in our numerical investigation.

One future research direction will be a further study of the principal angle and its impact on the convergence of the AC-ROM scheme. We will also investigate whether it plays a role in other popular schemes such as penalty methods. Another research direction that we plan to pursue is numerical stabilization of ROMs whose velocity-pressure ROM spaces do not satisfy the inf-sup/LBB condition. To our
knowledge, numerical stabilization to account for the violation of the LBB condition in ROMs has been investigated only in~\cite{caiazzo2014numerical}. In Section 3.2.3 in~\cite{caiazzo2014numerical}, it was shown that adding a pressure stabilizing/Petrov-Galerkin (PSPG) term to the ROM formulation yields better results than those produced by the other two velocity-pressure ROMs that were investigated.

	\bibliographystyle{plain}
	\bibliography{WorksCited,traian}
\end{document}